\PassOptionsToPackage{dvipsnames}{xcolor}
\documentclass[a4paper]{amsart}
\usepackage[utf8]{inputenc}
\usepackage[english]{babel}
\usepackage[autostyle]{csquotes}
\usepackage[square,numbers]{natbib}
\usepackage{amsmath}
\usepackage{amsthm}
\usepackage{array}
\usepackage{pgfplots}
\usepackage{amssymb}
\usepackage{graphicx}
\usepackage{caption}
\usepackage{sidecap}
\usepackage{mathtools}
\usepackage{mathrsfs}
\mathtoolsset{showonlyrefs}
\usepackage{enumerate}
\usepackage{enumitem}
\usepackage{esint}
\setlist[itemize]{leftmargin=5mm}
\setlist[enumerate]{leftmargin=*}
\setlist[description]{leftmargin=3mm, style=nextline}

\usepackage[a4paper,top=1.5cm, bottom=1.5cm, left=2cm, right=2cm]{geometry}

\usepackage{scalerel}
\usepackage{longtable}
\usepackage{multicol}
\usepackage{multirow}

\usepackage[colorlinks = true, urlcolor = blue, citecolor = PineGreen, linkcolor = blue]{hyperref}

\usepackage[bottom]{footmisc}

\newcommand{\R}{{\mathbb R}}
\newcommand{\N}{{\mathbb N}}

\newcommand{\Lc}{\mathcal{L}}

\newcommand{\weakly}{\rightharpoonup}

\newcommand{\de}{\partial}
\newcommand{\supp}{\mathrm{supp}\,}

\newcommand{\norm}[1]{\left\| #1 \right\|}

\newcommand{\iprod}[2]{\left\langle #1 ; #2 \right\rangle}

\DeclareMathOperator{\dist}{dist}

\DeclareMathOperator{\rank}{rank}
\DeclareMathOperator{\dive}{div}
\DeclareMathOperator{\curl}{curl}

\renewcommand{\phi}{\varphi}

\renewcommand{\ge }{\geq}
\renewcommand{\le }{\leq}
\renewcommand{\hat }{\widehat}

\newcommand{\haus}{\mathcal{H}}

\numberwithin{equation}{section}

	\newtheorem{theorem}{Theorem}[section]
	
	\newtheorem{proposition}[theorem]{Proposition}
	\newtheorem{lemma}[theorem]{Lemma}
	\newtheorem{corollary}[theorem]{Corollary}
	\newtheorem{conjecture}[theorem]{Conjecture}
	
	\newtheorem{question}[theorem]{Question}
	\newtheorem*{theorem*}{Theorem}
	\newtheorem*{proposition*}{Proposition}
	\newtheorem*{lemma*}{Lemma}
	\newtheorem*{corollary*}{Corollary}
	\newtheorem*{conjecture*}{Conjecture}
	\newtheorem*{statement*}{Statement}
	
	\newtheorem*{claim*}{Claim}
\theoremstyle{definition}
	\newtheorem{definition}[theorem]{Definition}
	\newtheorem{example}[theorem]{Example}
	
	\newtheorem*{definition*}{Definition}
	
	\newtheorem{hypo}[theorem]{Condition}
\theoremstyle{remark}
	\newtheorem{remark}[theorem]{Remark}

\makeatletter
\renewenvironment{proof}[1][\proofname]{\par
  \pushQED{\qed}%
  \normalfont \topsep6\p@\@plus6\p@\relax
  \trivlist
  \item[\hskip\labelsep
        \bfseries
    #1\@addpunct{.}]\ignorespaces
}{%
  \popQED\endtrivlist\@endpefalse
}
\makeatother

\newcommand*\dif{\mathop{}\!\mathrm{d}}
\newcommand{\Ac}{\mathcal{A}}
\newcommand{\Bc}{\mathcal{B}}

\newcommand{\Hc}[1]{\haus^{#1}}
\newcommand{\hu}{\haus^{d-1}}

\newcommand{\rest}{\llcorner}

\newcommand{\Sph}{\mathbb{S}}

\renewcommand{\Mc}{\mathcal{M}}

\newcommand{\face}[1]{\Gamma_{#1}}

\newcommand{\cinf}[2]{C^{\infty}(\R^{#1},\R^{#2})}
\newcommand{\ccinf}[2]{C_c^{\infty}(\R^{#1},\R^{#2})}
\newcommand{\ind}[1]{\mathbf{1}_{#1}}
\renewcommand{\l}{\ell}
\newcommand{\M}{\mathcal{M}}
\newcommand{\normu}[1]{\norm{#1}_{L^1}}

\newcommand{\op}[1]{\mathscr{Op}(#1)}

\DeclarePairedDelimiterX{\scal}[2]{\langle}{\rangle}{#1, #2}

\DeclareMathOperator{\spn}{span}

\usepackage{aligned-overset}
\usepackage{pgfplots}
\usepackage{tikz,tikz-3dplot}
\usepackage{float}
\usepackage{standalone}

\colorlet{luigi}{green!55!gray}
\colorlet{carlo}{blue!40!gray}
\DeclareMathOperator{\Id}{Id}
\usetikzlibrary {lindenmayersystems, arrows,decorations.markings,intersections,patterns,shadings,calc,shapes,positioning, bending,fadings, angles, pgfplots.fillbetween}
\pgfplotsset{compat=newest}
\usepgfplotslibrary{fillbetween}

\newcommand{\Pc}{\mathcal{P}}
\newcommand{\Fc}{\mathcal{F}}

\title{Non-rigidity of the absolutely continuous part of $\mathcal{A}$-free measures}

\author[L.~De Masi]{Luigi De Masi}
\address{\textit{L.~De Masi:} Dipartimento di Matematica, Università di Trento, Via Sommarive 14, I–38123 Trento, Italy.}
\email{luigi.demasi@unitn.it}

\author[C.~Gasparetto]{Carlo Gasparetto}
\address{\textit{C.~Gasparetto:} Dipartimento di Matematica, Università di Pisa, Largo Bruno Pontecorvo 5, I–56127 Pisa, Italy.}
\email{carlo.gasparetto@dm.unipi.it}


\begin{document}
\bibliographystyle{dinat}
\maketitle

\begin{abstract}
We generalize a result by Alberti, showing that, if a first-order linear differential operator $\Ac$ belongs to a certain class, then any $L^1$ function is the absolutely continuous part of a measure $\mu$ satisfying $\Ac\mu=0$. When $\Ac$ is scalar valued, we provide a necessary and sufficient condition for the above property to hold true and we prove dimensional estimates on the singular part of $\mu$. Finally, we show that operators in the above class satisfy a Lusin-type property.
\end{abstract}

\section{Introduction and main results}\label{sec:intro}
In \citep[Theorem 3]{alberti1991lusin} it was shown that any $L^1$ vector field is the absolutely continuous component of the distributional gradient of some $SBV$ function. In other words, given $f\in L^1(\R^d,\R^d)$, there exists $u\in L^1(\R^d)$ and a $(d-1)$-rectifiable finite measure $\sigma$ with values in $\R^d$ such that
\begin{equation}
    Du = f+\sigma,
    \qquad
    |\sigma|(\R^d)\leq C \norm{f}_{L^1(\R^d)},
\end{equation}
where the equation holds in the sense of distributions and the constant $C$ depends only on $d$.

Since any distributional gradient $Du$ satisfies $\curl Du=0$ in the sense of distributions,\footnote{Where $(\curl Du)_{ij} =\de_j u_i-\de_i u_j$ for $1\leq i < j \leq d$.} the condition $Du = f+\sigma$ can be written as $\curl (f +\sigma)=0$.

The question we address in this paper is whether one can prove a similar result if $\curl$ is replaced by some other differential operator:
\begin{question}\label{prb:alberti}
	Let $\Ac\colon \cinf{d}{m}\to\cinf{d}{n}$ be a first-order linear differential operator and $k \in \{0, 1,\dots,d-1\}$. Is it true that for any $f\in L^1(\R^d,\R^m)$ there exists a $k$-rectifiable measure $\sigma$ with values in $\R^m$ such that
	\begin{equation}
		\Ac(f+\sigma)=0 ,
  \qquad
    |\sigma|(\R^d)\leq C \norm{f}_{L^1(\R^d)},
\end{equation}
where the equation holds in the sense of distributions and the constant $C$ depends only on $\Ac$?
\end{question}
We call \emph{$k$-balanceable} those operators for which Question \ref{prb:alberti} has a positive answer.
While \citep[Theorem 3]{alberti1991lusin} gives a positive answer to the above question with $k=d-1$ when $\Ac=\curl$, it is clear that the answer is negative for at least some first-order linear differential operator.
For instance, in the cases $\Ac\phi = D\phi$ or $\Ac\phi = (\curl\phi,\dive\phi)$, it is well-known that a measure satisfying $\Ac\mu=0$ is the distribution induced by a smooth function $f$ that satisfies $\Ac f=0$ in the classical sense: in these cases, Question \ref{prb:alberti} has a negative answer.
In general, the structure of the singular part of a \textit{$\Ac$-free measure} largely depends on the wave cone associated to $\Ac$, as was noted in \citep{dph_rindler_afree_original} (see also \citep{dph_rindler_proceedings, arroyo-rabasa_Lusin, arroyo_rabasa_Young, arroyorabasa2023higher}).

The first main result of the present work characterizes scalar-valued first-order linear differential operators for which Question \ref{prb:alberti} has a positive answer, namely those operators $\Ac\colon \cinf{d}{m}\to\cinf{d}{}$ of the form
\begin{equation}\label{eq:first_order_operator_type}
	\Ac f(x) = \sum_{i=1}^d\sum_{j=1}^mA_{ij}\de_{i}f_j(x)
\end{equation}
for every $f\in \cinf{d}{m}$, where $A=(A_{ij}) \in \R^{d \times m}$.
For instance, the divergence operator is represented as above with the choice $A=\Id$.
\begin{theorem}\label{thm:main_scalar}
	Let $\Ac$ be a scalar-valued first order linear differential operator as in \eqref{eq:first_order_operator_type} with $r \coloneqq \rank A>0$. The following hold:
\begin{enumerate}[label=\alph*)]
\item\label{item:main_positive} $\Ac$ is $k$-balanceable for every $k \geq d+1-r$;
\item\label{item:main_sharpness} there exists $f\in L^1(\R^d,\R^m)$ such that $\Ac(f+\sigma)\neq 0$ for every Radon measure $\sigma$ satisfying\footnote{Here $\Hc{s}$ denotes the $s$-dimensional Hausdorff measure on $\R^d$ and $\sigma\perp\mu$ means that there is some $E\subset \R^d$ such that $\sigma(\R^d\setminus E)=0$ and $\mu(E)=0$.} $|\sigma| \perp \Hc{d+1-r}$. In particular $\Ac$ is not $k$-balanceable for any $k \leq d-r$.
	\end{enumerate}
In particular, if $\rank A=1$, then Question \ref{prb:alberti} has a negative answer for any $k\in\{0,1,\dots,d-1\}$.
\end{theorem}

A quick summary of the proof of Theorem \ref{thm:main_scalar} is given at the beginning of Section \ref{sec:balance_scalar}.

\begin{remark}[$0$-balanceable operators]\label{rem:0-balanceable}
    Part \ref{item:main_sharpness} of Theorem \ref{thm:main_scalar} in particular implies that the only $0$-balanceable operator is the null one. Indeed, if $r \coloneqq \rank A \neq 0$, then $0 \leq d - r$.   
\end{remark}

\begin{remark}\label{rem:sharp_sharpness}
   One cannot replace part \ref{item:main_sharpness} of Theorem \ref{thm:main_scalar} with the stronger
    \begin{equation}
        \Ac(f+\sigma)= 0
        \implies
        |\sigma| \ll \Hc{d+1-r}.
    \end{equation}
    The reason for this is the following: as shown by Example \ref{ex:sharp_sharpness}, $\Ac \mu = 0$ does not imply $|\mu| \ll \Hc{d+1-r}$; for such $\mu$ the relation $\Ac(f+\sigma)=0$ is equivalent to $\Ac(f+\sigma +\mu)=0$ and $\sigma + \mu$ is in general not absolutely continuous with respect to $\Hc{d+1-r}$.
\end{remark}

The second part of the present work concerns vector-valued first order operators, namely those operators of the form $\Ac=(\Ac^1,\dots,\Ac^n)$, where each $\Ac^i$ is a scalar-valued operator, as for $\Ac = \curl$. In this framework, answering Question \ref{prb:alberti} becomes substantially harder.
The reason is that, in this case, the condition $\Ac(f+\sigma)=0$ is a \emph{system} of partial differential equations. Although Theorem \ref{thm:main_scalar} provides a solution for each equation, in general these measures are not solutions for the other equations of the system.

However, it is possible to build a solution of the system $\Ac(f+\sigma)=0$, provided the matrices $A^i$ which define each scalar-valued component $\Ac^i$ of $\Ac$ satisfy some algebraic properties which relate each other, see Condition \ref{hyp:solving_2}. Since its formulation is quite technical, we refer the reader to Section \ref{sec:vector_valued}; here we mention that key parts are a ``shared $\rank$-$2$" property \eqref{eq:hyp_suff_2a} and a ``combined antisymmetry" \eqref{eq:hyp_suff_2c} of matrices 
$A^i$ (see also Remark \ref{rem:intuition_condition} for more intuitions on Condition \ref{hyp:solving_2}). Exploiting these algebraic properties, we can produce $(d-1)$-rectifiable measures which balance $\Ac$.
This is precisely the content of the second main result of this paper:
\begin{theorem}\label{thm:main_vector_valued}
	Let $\Ac\colon \cinf{d}{m}\to\cinf{d}{n}$ be a first-order linear differential operator satisfying Condition \ref{hyp:solving_2}; then $\Ac$ is $(d-1)$-balanceable.
\end{theorem}

It is easily checked (see Proposition \ref{prop:scalar_condition_hyp}) that a scalar operator of the form \eqref{eq:first_order_operator_type} satisfies Condition \ref{hyp:solving_2} if and only if 
$A=0$ or $\rank A\ge2$, consistently with Theorem \ref{thm:main_scalar}.

Furthermore, we show in Section \ref{sec:exterior_derivative} that the exterior derivative acting on differential forms satisfies Condition \ref{hyp:solving_2}.
In particular, we prove the following:
\begin{corollary}\label{cor:diff_forms}
    Let $p\in\{1,\dots,d\}$ and let $f\in L^1(\R^d,\Lambda^{p}(\R^d))$ be a differential form of degree $p$.
    Then there exists a $(d-1)$-rectifiable, differential-form-valued measure $\sigma$ such that
    \begin{equation}
        \mathrm{d}(f+\sigma)=0\qquad\mbox{and}\qquad |\sigma|(\R^d)\le C||f||_{L^1(\R^d)}.
    \end{equation}
\end{corollary}
In particular, with the choice $p=1$, Corollary \ref{cor:diff_forms} recovers Alberti's original result \citep[Theorem 3]{alberti1991lusin}.


The third main result of the present paper is a Lusin-type property which generalizes \citep[Theorem 1]{alberti1991lusin}:
\begin{theorem}\label{thm:Lusin_main}
    Let $\Ac\colon \cinf{d}{m}\to\cinf{d}{n}$ be a first-order linear differential operator satisfying Condition \ref{hyp:solving_2} and let $f \in L^1(\Omega,\R^m)$, where $\Omega\subset\R^d$ is an open set with finite measure; for every $\varepsilon>0$ there exist an open set $U \subset \Omega$ and a function $h \in C^0(\Omega,\R^m)$ with the following properties:
    \begin{subequations}\label{eq:lusin}
    \begin{gather}
    \Lc^d(U) < \varepsilon\Lc^d(\Omega);\label{eq:lusin_a}
    \\[2pt]
    f=h \qquad \text{in } \Omega \setminus U;\label{eq:lusin_b}
    \\[2pt]
    \Ac h=0\qquad\text{in }\Omega;\label{eq:lusin_c}
    \\[2pt]
    \norm{h}_{L^p(\Omega)} \leq C \varepsilon^{\frac{1}{p}-1} \norm{f}_{L^p(\Omega)}\label{eq:lusin_d}
    \qquad
    \forall p \in [1,+\infty],
    \end{gather}
    \end{subequations}
where the constant $C$ depends only on $\Ac$.
\end{theorem}
As above, it is worth remarking that $\Ac=\curl$ and any scalar operator of the form \eqref{eq:first_order_operator_type} with $\rank A\ge2$ satisfy Condition \ref{hyp:solving_2}.
The function $h$ in the statement of Theorem \ref{thm:Lusin_main} is built, in our proof, via an appropriate mollification of the measure $\sigma$ given by Theorem \ref{thm:main_vector_valued}. This choice is clearly not unique.

As previously stated, our result is related to the line of work initiated in \citep{dph_rindler_afree_original} of determining the structure of $\Ac$-free measures, that is measures $\mu$ satisfying $\Ac\mu=0$ in the sense of distributions.
While most works in this direction (for instance, \citep{dph_rindler_afree_original,arroyo2019dimensional})
focus on determining what singular measures are admissible as singular parts of $\Ac$-free measures, our results address rigidity and non-rigidity properties of the absolutely continuous part of $\Ac$-free measures.

A similar  question has been investigated in \citep{arroyo-rabasa_Lusin}: given any differential operator $\Bc$, one asks what functions are the absolutely continuous part of $\Bc u$ for some function $u$, thus dealing with measures in the \emph{image} of the operator $\Bc$, while in our work we study the non-rigidity of the absolutely continuous part of measures in the \emph{kernel} of an operator $\Ac$.
By the above mentioned connection between the range of the gradient operator and the kernel of $\curl$, \citep{arroyo-rabasa_Lusin} also generalizes \citep[Theorem 3]{alberti1991lusin}, although with different techniques.
We moreover mention that, while the approach of \citep{arroyo-rabasa_Lusin} is intrinsically limited to adding a $(d-1)$-rectifiable measure $\sigma$ to a function $f\in L^1$ in order to solve $f +\sigma = \Bc u$ for some $u$, our setting allows for more singular rectifiable measures.
Compare for instance \citep[Corollary 6]{arroyo-rabasa_Lusin}, which states that divergence is $(d-1)$-balanceable, with Theorem \ref{thm:main_scalar}, which proves that divergence is $k$-balanceable for every $k \in \{1,\dots,d-1\}$.
One should also see \citep{raita_potentials} for more connections between our point of view and that of \citep{arroyo-rabasa_Lusin}.

In \citep{AlbMar22}, a problem related to Question \ref{prb:alberti} is studied in the framework of \emph{flat} and rectifiable currents: in \citep[Proposition 3.3]{AlbMar22} the authors prove that every $k$-dimensional flat chain with finite mass has the same boundary of a $k$-rectifiable current.

Finally, we remark that Question \ref{prb:alberti} is worth-studying also if one simply requires $\sigma \perp \Lc^d$, without any dimensional constraint on $\sigma$.

The rest of the paper is structured as follows. In Section \ref{subsec:notations}, we collect most notation used thoughout the paper and we state and prove some simple results concerning first-order linear differential operators.
Sections \ref{sec:balance_scalar}, \ref{sec:vector_valued} and \ref{sec:lusin} are dedicated to the proofs of Theorems \ref{thm:main_scalar}, \ref{thm:main_vector_valued} and \ref{thm:Lusin_main}, respectively. Section \ref{sec:examples_and_obs} is dedicated to collecting some examples and observations on Condition \ref{hyp:solving_2}.

\section{Notation and reduction of the problem}\label{subsec:notations}

Here we collect in a brief and mostly schematic form the notation and definitions used throughout the paper.

\begin{longtable}{@{\extracolsep{\fill}}lp{0.75\textwidth}}
\multicolumn{2}{c}{\textbf{General notation}}\\[4pt]
$B_r(x)$                &   open ball of radius $r$ centered at $x$; when $x=0$, we omit its indication;\\
$\mathbb{S}^{d-1}$		& 	unit sphere in $\R^d$, namely the set $\{x \in \R^d \colon |x|=1\}$;\\
$\nu_E(x)$              &   outer unit normal to a set $E$ at $x \in \de E$;\\
$\dive_\Sigma X$		&	tangential divergence of $X \in C^1(\R^d,\R^d)$ with respect to $T_x \Sigma$, where $\Sigma \subset \R^d$ is a $k$-dimensional manifold of class $C^1$, namely $\dive_\Sigma X(x) = \sum_{j=1}^k \de_{\tau_j} X(x) \cdot \tau_j$, where $\{\tau_j\}_{j=1}^k$ is an orthonormal basis of $T_x \Sigma$.\\
$\mu \perp \sigma$ 		& 	mutually singular Borel measures $\mu,\sigma$, namely for which there exists $E \subseteq \R^d$ such that $|\mu|(E) = |\sigma|(\R^d \setminus E)=0$;\\
$\mu \ll \sigma$		&	the measure $\mu$ is absolutely continuous with respect to $\sigma$, i.e. $\mu(E)=0$ if $\sigma(E)=0$;\\
$\M(U,\R^m)$		     &	the space of  finite Radon measures on $U\subset\R^d$ with values in $\R^m$;\\
$\M^k (U,\R^m)$		& 	the space of $k$-rectifiable finite Radon measures on $U$ with values in $\R^m$, namely the set of measures $\mu \in \M(U,\R^m)$ for which there exist a $k$-rectifiable set $E \subseteq U$ and  $\theta \in L^1_{\Hc{k}}(E,\R^m)$ such that $\mu= \theta \Hc{k} \rest E$.\\[12pt]
\multicolumn{2}{c}{\textbf{}}\\[4pt]
\multicolumn{2}{c}{\textbf{Cubes }}\\[4pt]
$q(Q)$                  & 
center of the cube $Q$; \\
$\Pc_\ell(Q)$           &
dyadic decomposition of $Q=Q_r(x)$ at level $\ell$, namely the collection of $2^{d\ell}$ cubes with faces parallel to those of $Q$, with side length $2^{-\ell}r$, which cover $Q$; \\
$\face{i \pm}$			&
the face of $Q=(-1,1)^d$ where the exterior unit normal is $\nu_Q = \pm e_i$, namely $\{x\colon x_i=\pm1\mbox{ and }|x_j|<1\mbox{ for }j\neq i\}$\\
\multicolumn{2}{c}{\textbf{}}\\[4pt]
\multicolumn{2}{c}{\textbf{Operators}}\\[4pt]
$\op{d,m,n}$            &   space of first order linear differential operators $\Ac:C^\infty(\R^d,\R^m)\to C^\infty(\R^d,\R^n)$;\\
$(A^1,\dots,A^n)$       &   operator  $\Ac \in \op{d,m,n}$ represented by $A^1,\dots,A^n \in \R^{d \times m}$, namely $\Ac f(x) = \sum_{i,j}\big(A^1_{ij}\de_if_j(x),\dots,A^n_{ij}\de_i f_j(x)   \big)$;\\
$\Ac \cong \Bc$         &   operators $\Ac=(A^1,\dots,A^n)$ and $\Bc=(B^1,\dots,B^h)$ such that $\spn(A^1,\dots,A^n)= \spn(B^1,\dots,B^h)$.
\end{longtable}

We say that $\Ac \in \op{d,m,n}$ is \emph{scalar-valued} if $n=1$, otherwise that it is \emph{vector-valued}.
We moreover recall the following definition, which we already stated in the introduction.
\begin{definition}[$k$-balanceable operator]\label{def:balanceable}
	Given $k \in \N$, we say that a first order linear differential operator $\Ac\colon\cinf{d}{m}\to\cinf{d}{n}$ is \textit{$k$-balanceable} if there is $C>0$ such that, for every $f\in L^1(\R^d,\R^m)$, there is a $k$-rectifiable measure $\sigma$ in $\R^d$ with values in $\R^m$ such that
	\begin{equation}\label{eq:def_balanceable}
		\Ac(f+\sigma)=0\qquad\mbox{and}\qquad|\sigma|(\R^d)\le C\norm{f}_{L^1(\R^d)}.
	\end{equation}
\end{definition}

We now turn our attention to Question \ref{prb:alberti}.
The main result of this section is the fact that, in order to prove that an operator $\Ac$ is $k$-balanceable, it is sufficient to show that one can balance $v_j \ind{E_j}$, for some basis $\{v_j\}_j$ of $\R^m$ and some bounded ``almost closed" sets $E_j \subset \R^d$ of positive measure.
\begin{proposition}\label{prop:reduction_cubes_basis}
Let $\Ac\in \op{d,m,n}$ and let $\Fc = \{v_1,\dots,v_m\}$ be a basis of $\R^m$; let us assume that, for each $v \in \Fc$, there exists a bounded Borel set $E \subset \R^d$ and $\sigma \in\Mc^k(\R^d,\R^m)$ such that
\begin{equation}\label{eq:conditions_reduction_cubes}
\Lc^d(E)>0, \qquad \Lc^d(\bar E \setminus E)=0,
\qquad
		\Ac(v\ind{E} + \sigma) = 0.
\end{equation}
Then $\Ac$ is $k$-balanceable.
\end{proposition}

\begin{proof}[Proof]
	
	The proof is inspired to the one of \citep[Theorem 3]{alberti1991lusin}.
	We need to prove that, for any $f\in L^1(\R^d,\R^m)$, there exists $\mu\in\Mc^k(\R^d,\R^m)$ such that
	\begin{equation}\label{eq:reduction_goal}
		\Ac(f+\mu) = 0,
  \qquad
    |\mu|(\R^d)\leq C \norm{f}_{L^1(\R^d)},
	\end{equation}
 with $C$ independent of $f$. Arguing component-wise, we can fix $v \in \Fc$ and prove the validity of \eqref{eq:reduction_goal} for any $f \in L^1(\R^d,\spn(v))$.

 By $\Lc^d(\bar E \setminus E)=0$, it follows $\ind{\bar E}=\ind{E}$ $\Lc^d$-a.e., thus the measure $\sigma$ in the statement satisfies $\Ac(v\ind{\bar E} + \sigma)=0$ as well. Therefore we can assume that $E$ is closed. Moreover, up to replacing $\sigma$ with an appropriate translation and rescaling, the properties in the statement are valid for $\lambda v \ind{x+rE}$ for every $\lambda \in \R$, every $x \in \R^d$ and every $r>0$. In particolar, in the following we assume  that $0 \in E$, that $E\subset B_1$ and that $|v|=1$.

	Towards \eqref{eq:reduction_goal}, we aim at defining following sequences:
	\begin{itemize}
		\item a sequence of functions $\{f_\ell\}_{\ell\in\N}$
		such that
		\begin{equation}\label{eq:f_j}
			\normu{f-f_\ell} \leq 2^{-\ell} \normu{f};
		\end{equation}
		\item a sequence of $k$-rectifiable measures $\{\mu_\ell\}_{\ell\in\N} \subset \M(\R^d,\R^m)$ satisfying
		\begin{equation}\label{eq:w_j}
			\Ac \left (
			f_\ell
			+
			\sum_{i=1}^\ell\mu_i
			\right )
			=
			0
			\qquad
			\text{and}
			\qquad
			|\mu_\ell|(\R^d)
			\leq
			2^{-\ell}\tilde{C}\norm{f}_{L^1},
		\end{equation}
		where $\tilde{C}>0$ is a constant independent of $f$.
	\end{itemize}
	By 	\eqref{eq:f_j} and \eqref{eq:w_j}, 
	$\mu \coloneqq \sum_{i=1}^{+\infty} \mu_i$ satisfies \eqref{eq:reduction_goal}.
	
	\begin{description}
		\item[$\bullet$ Definition of $f_0,\mu_0$]
		
		We set $f_0 \equiv 0$ and $\mu_0 = 0$. Clearly $f_0,\mu_0$ satisfy \eqref{eq:f_j} and \eqref{eq:w_j}.
		
		\item[$\bullet$ Definition of $f_\ell,\mu_\ell$ satisfying \eqref{eq:f_j} and \eqref{eq:w_j}]
		
		Let us assume that $f_{\ell-1}$ and $\mu_1,\dots,\mu_{\ell-1}$ are defined and satisfy \eqref{eq:f_j} and \eqref{eq:w_j}.
Let us fix an open ball $B$ such that
\begin{equation}\label{eq:norm_ball_big}
    \norm{f-f_{\ell-1}}_{L^1(\R^d \setminus B)} \leq 2^{-\ell-1} \norm{f}_{L^1(\R^d)}.
\end{equation}
If $x \in \R^d$ is a Lebesgue point of $u\coloneqq f-f_{\ell-1}$, since $x+rE\subset B_r(x)$, it holds
  \begin{equation}
        \begin{split}
            \limsup_{r \to 0} \fint_{x+rE} |u(y)-u(x)| \dif \Lc^d(y)
        \leq
        \limsup_{r \to 0}
        \frac{\Lc^d(B_r(x))}{\Lc^d(x+rE)} \fint_{B_r(x)} |u(y)-u(x)| \dif \Lc^d(y)
        =
        0.
        \end{split}
    \end{equation}
    Thus, for every Lebesgue point $x \in B$ of $u$, there exists $\rho_x \in \big(0,\dist(x,\de B)\big)$ such that
    \begin{equation}\label{eq:mean_difference_small}
        \frac{1}{\Lc^d(x+rE)} \int_{x +rE} |u(y)-u(x)| \dif \Lc^d(y)
        <
        \frac{2^{-\ell-1} \norm{f}_{L^1(\R^d)}}{\Lc^d(B)}
        \qquad
        \forall r \in (0,\rho_x).
    \end{equation}
    Now let us call $\tilde B$ the set of Lebesgue points of $u$ in $B$ and let us consider the covering $\Fc$ of $\tilde B$ given by
    \begin{equation}
        \Fc \coloneqq \left \{ x+rE \colon x \in \tilde B, \,\, r \in (0,\rho_x) \right\}.
    \end{equation}
    By the version of Vitali covering theorem for arbitrary closed sets given in \citep[Theorem 2.8.17]{federer2014geometric} applied to the fine covering $\Fc$,
there exists a countable sub-family $\Fc_\ell \coloneqq\{F_i=x_i+r_i E\}_{i\in\N} \subset \Fc$ where the $F_i$ are mutually disjoint, such that
	\begin{equation}\label{eq:vitali_convering_measure}
		0= \Lc^d\left(\tilde B \setminus\bigcup_{i \in \N}F_i\right)
  =
  \Lc^d\left(B \setminus\bigcup_{i \in \N}F_i\right),
	\end{equation}
 where the last equality follows from the Lebesgue theorem.
 We now define
 \begin{equation*}
			\phi_\ell(x) = \sum_{i \in \N}\big(f(x_i) - f_{\ell-1}(x_i) \big)\ind{F_i}(x).
\end{equation*}
    By the definition of $\Fc$, it holds $\phi_\ell=0$ on $\R^d \setminus B$, thus
    \begin{equation}\label{eq:estimate_norm_f-f_ell}
     \begin{split}
            \norm{f - f_{\ell - 1} - \phi_\ell}_{L^1(\R^d)}
        & =
        \norm{f - f_{\ell - 1}}_{L^1(\R^d \setminus B)}
        +
        \norm{f - f_{\ell - 1} - \phi_\ell}_{L^1(B)}
        \\
        \overset{\eqref{eq:norm_ball_big}, \eqref{eq:vitali_convering_measure}}&{\leq}
        2^{-\ell-1} \norm{f}_{L^1(\R^d)} + \sum_{i \in \N} \int_{F_i} |u(y) - u(x_i)| \dif \Lc^d(y)
        \\
        \overset{\eqref{eq:mean_difference_small}}&{\leq}
        2^{-\ell-1} \norm{f}_{L^1(\R^d)} \left( 1 + \frac{1}{\Lc^d(B)}  \sum_{i \in \N} \Lc^d(F_i)   \right)
        \\
        & =
        2^{-\ell} \norm{f}_{L^1(\R^d)},
     \end{split}
    \end{equation}
    where the last equality is a consequence of  \eqref{eq:vitali_convering_measure} together with the fact that the sets $F_i$ are mutually disjoint and that $F_i \subseteq B$.
		Defining $f_\ell\coloneqq f_{\ell-1} +\phi_\ell$, by \eqref{eq:estimate_norm_f-f_ell} we obtain immediately $\normu{f - f_{\ell}} \leq 2^{-\ell}\normu{f}$, proving \eqref{eq:f_j}.
  
  In order to show the existence of $\mu_\ell$ satisfying \eqref{eq:w_j}, we first estimate
\begin{equation}\label{eq:estimate_norm_phi_ell}
    \begin{split}
        \norm{\phi_\ell}_{L^1(\R^d)}
    \leq
    \norm{f- f_{\ell-1} - \phi_\ell}_{L^1(\R^d)}+ \norm{f- f_{\ell-1}}_{L^1(\R^d)} 
    \leq
    2^{-\ell+2} \norm{f}_{L^1(\R^d)},
    \end{split}
\end{equation}
where the last inequality follows by \eqref{eq:estimate_norm_f-f_ell} and by inductive hypothesis.
		
		Fix now $i \in \N$ and let $\bar\phi_i \coloneqq f(x_i) - f_{\ell-1}(x_i)$.
		Let also $\tau(x)=x_i +r_ix$ be the affine map such that $\tau(E)=F_i$.
We define  
\begin{equation}
    \mu_{F_i} \coloneqq r_i^d (\bar \phi_i \cdot v)  (\tau)_{\#}\sigma
\end{equation}
which, by $\Ac(v\ind{E}+\sigma)=0$ and by linearity, satisfies
		\begin{equation}\label{eq:balance_small_cubes}
			\Ac\bigg(\bar\phi_{i}\ind{F_i}+\mu_{F_i}\bigg) = r_i^d (\bar \phi_i \cdot v)\,\Ac\bigg((\tau)_{\#}(v\ind{E}+\sigma)\bigg)=0.
		\end{equation}
		Moreover,
		\begin{equation}\label{eq:estimate_l_1_small_cubes}
			|\mu_{F_i}|(\R^d) = r_i^d|\bar \phi_i||\sigma|(\R^d) \le C_1 |\bar \phi_i|\Lc^d(F_i),
		\end{equation}
		where $C_1$ depends only on $\Ac$.

	Therefore, defining $\mu_\ell\coloneqq \sum_{i \in \N}\mu_{F_i}$
	and summing \eqref{eq:balance_small_cubes} and \eqref{eq:estimate_l_1_small_cubes} for $i \in \N$, we obtain
	\begin{gather}
		\Ac(\varphi_\ell + \mu_\ell)=0,\\[4pt]
		|\mu_\ell|(\R^d)
		\leq
		\sum_{i \in \N} |\mu_{F_i}|(\R^d)
		\le C_1 \sum_{i \in \N} \Lc^d(F_i)|\bar\phi_i|
		=
		C_1||\phi_\ell||_{L^1(\R^d)}
		\overset{\eqref{eq:estimate_norm_phi_ell}}{\leq}
		\tilde{C} 2^{-\ell}\norm{f}_{L^1(\R^d)},
	\end{gather}
	where $\tilde C=4C_1$ depends only on $\Ac$.
	Thus, since $f_\ell=f_{\ell-1}+\phi_\ell$ and, by inductive hypothesis, it holds $\Ac(f_{\ell-1} + \sum_{i=1}^{\ell-1} \mu_i)=0$, we have
	\begin{equation}
		\Ac\left(f_\ell + \sum_{i=1}^\ell \mu_i\right)
		=
		0.
	\end{equation}

	\item[$\bullet$ Convergence]
	By \eqref{eq:w_j}, the sequence of the partial sums of $\sum_{j \in \N} \mu_j$ is a Cauchy sequence in the strong topology of $\M(\R^d,\R^m)$.
	Hence it converges, in the same topology, to a measure $\mu \in \M(\R^d,\R^m)$ which is $k$-rectifiable (by strong convergence) and which satisfies, by \eqref{eq:w_j},
	\begin{equation}
		\Ac(f + \mu)=0,
		\qquad
		|\mu|(\R^d) \leq \tilde{C} \normu{f}.
	\end{equation}
	\end{description}

\end{proof}

\begin{remark}
    We underline the fact that the same proof works in case one merely requires $\sigma \perp \Lc^d$, again by strong convergence of the series $\sum_{j \in \N} \mu_j$ where each $\mu_j \perp \Lc^d$.
\end{remark}

We conclude this section with a simple remark, whose proof we omit, which we will use in the rest of the paper.
\begin{lemma}\label{lmm:Amu_dive}
Let $\Ac=(A^1,\dots,A^n) \in \op{d,m,n}$, where $A^k \in \R^{d \times m}$ and let $\sigma \in \M(\R^d,\R^m)$ be such that $\sigma= g |\sigma|$, where $g \colon \R^d \to \R^m$. Then
\begin{equation}
\Ac \sigma = \left ( \dive (A^1g |\sigma|), \dots, \dive (A^ng |\sigma|) \right ).
\end{equation}
\end{lemma}

\section{Balancing scalar operators}\label{sec:balance_scalar}

In this section we prove Theorem \ref{thm:main_scalar}. The proof is split in two parts.
Both parts rest on the fact that any scalar valued operator $\Ac$ can be written, up to choosing appropriate bases in $\R^d$ and $\R^m$, as a \enquote{lower dimensional} divergence, namely in the form $\dive_r f\coloneqq\de_1f_1+\dots\de_r f_r$, with $r=\rank A$. 
Geometrically, given $f\in\R^m$ and a cube $Q$ with a pair $\Gamma_{\pm}$ of faces perpendicular to $Af$, the equation $\Ac(f_1\ind{Q} + \sigma)=0$ means that
$A \sigma$ ``transports" the mass of $\Gamma_+$ into that of $\Gamma_-$.

In order to prove part \ref{item:main_positive} of Theorem \ref{thm:main_scalar}, given $k\ge d+1-r$, we first produce a $1$-rectifiable measure $\mu$ that transports the mass of a face of a $(d+1-k)$-dimensional cube to the mass of the opposite face. This is done in Lemma \ref{lmm:1_rect_dive}, see also Figure \ref{fig:ramified_divergence}. Then, the desired measure $\sigma$ is an appropriate rescaling of $\sigma = \mu\times\Hc{k-1}\rest(-1,1)^{k-1}$.

In order to prove part \ref{item:main_sharpness} of Theorem \ref{thm:main_scalar}, we show that any $\sigma$ satisfying
\begin{equation}
    \dive_r(f_1\ind Q+\sigma)=0
\end{equation}
admits a disintegration $\lambda\otimes\sigma_{x''}$  where, $\lambda \not \perp \Hc{d-r}$ and, for $\lambda$-almost every $x'' \in \R^{d-r}$, $\dive_r\sigma_{x''}$ is a measure. By \citep{arroyo2019dimensional}, this yields that $\sigma_{x''} \not \perp \Hc{1}$. A coarea-type inequality (Lemma \ref{lmm:fubini_geq_hausdorff}) provides $|\sigma| \not \perp \Hc{d-r+1}$, as desired.

\subsection{Proof of part \ref{item:main_positive} of Theorem \ref{thm:main_scalar} }

In the proof of part \ref{item:main_positive} of Theorem \ref{thm:main_scalar}, we are going to use the following lemma, illustrated in Figure \ref{fig:ramified_divergence}. The result seems to be already well known; however, since we could not find an explicit reference in the literature, we write both the statement and its proof below.

\begin{lemma}\label{lmm:1_rect_dive}
Let $h \in \N$. There exists a finite $1$-rectifiable measure $\mu \in \M^1((-1,1)^{h+1},\R^{h+1})$ such that
\begin{equation}\label{eq:1_rect_dive_face}
\dive \mu = \Hc{h} \rest \face{1+} - \Hc{h} \rest \face{1-},
\end{equation}
where $\face{1\pm}= \{\pm 1\} \times (-1,1)^{h}$.
\end{lemma}

\begin{figure}
    \centering
    \includegraphics{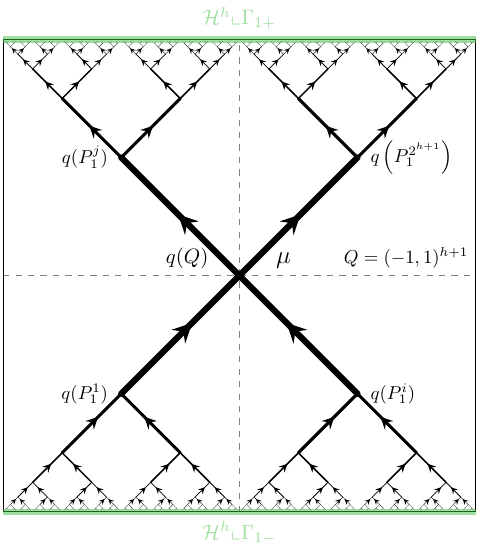}
    \caption{The measure $\mu$ given by Lemma \ref{lmm:1_rect_dive}, where the center $q(Q)$ of $Q$, the cubes in $\mathcal{P}_1(Q)$ and their centers are represented.}
    \label{fig:ramified_divergence}
\end{figure}

\begin{proof}
We start by fixing some notation.
Throughout the proof, for any two points $p,q\in\R^{h+1}$, we define the measure
\begin{equation*}
	\mu{[p,q]} = \frac{q-p}{|q-p|}\haus^1\rest[p,q],
\end{equation*}
where ${[p,q]}$ denotes the segment in $\R^{h+1}$ that joins $p$ and $q$. Notice that
\begin{equation}\label{eq:measure_transp}
	\dive{\mu[p,q]} = \delta_q-\delta_p\qquad\mbox{ and }\qquad\big|\mu[p,q]\big|(\R^{h+1}) = |q-p|.
\end{equation}
We claim there exists $\gamma\in\M^1((-1,1)^{h+1},\R^{h+1})$ such that
\begin{equation}\label{eq:claim_dive_gamma}
	\dive\gamma = \frac{1}{\Hc{h}(\Gamma_{1+})}\Hc{h} \rest \face{1+} - \delta_0.
\end{equation}
Given $\gamma$ as above, the measure $\mu$ satisfying the conclusion of the Lemma is just $\mu=\Hc{h}(\Gamma_{1+})\big(\gamma-\gamma'\big)$, where $\gamma'$ is the reflection of $\gamma$ across the origin, that is the push-forward of $\gamma$ through the map $\xi(x)\coloneqq -x$.

Let now $Q = (-1,1)^{h}\subset\R^{h}$. In order to define $\gamma$ as in the above claim, we inductively build $\{\gamma_\ell\}_{\ell\in\N}$ as follows.
\begin{itemize}
    \item We define
    \begin{equation}
        \gamma_1 \coloneqq \frac{1}{2^h}\sum_{P \in \Pc_1(Q)} \mu{\left[0,\left(\frac{1}{2}, q(P) \right) \right]},
    \end{equation}
    where
    $ \Pc_1(Q)$ is the dyadic decomposition of $Q$ at level $1$ defined in section \ref{subsec:notations}, $q(P)\in Q$ is the center of $P$ and the notation $(1/2,q(P)) = \{1/2\}\times \{q(P)\}\in(-1,1)^{h+1}$ was used.
    By \eqref{eq:measure_transp} it holds
    \begin{equation}
        \dive \gamma_1 = -\delta_0 + \frac{1}{2^h}\sum_{P \in \Pc_1(Q)} \delta_{\left(\frac{1}{2}, q(P) \right)}.
    \end{equation}
    \item Assume $\gamma_\ell$ is defined and satisfies
    \begin{equation}\label{eq:diramation_inductive}
		\dive\gamma_\ell = -\delta_0 + \frac{1}{2^{\ell h}}\sum_{P \in \Pc_\ell(Q)}\delta_{\left(1-2^{-\ell}, q(P) \right)}.
    \end{equation}
    We define $\gamma_{\ell+1}$ as
    \begin{equation}
        \gamma_{\ell+1} = \gamma_\ell + \frac{1}{2^{(\ell+1)h}} \sum_{P \in \Pc_\ell(Q)} \sum_{T \in \Pc_1(P)} \mu \left[ \left(1-2^{-\ell}, q(P) \right), \left(1-2^{-\ell-1}, q(T)\right) \right].
    \end{equation}
    It then holds
    \begin{equation}
    \begin{split}
    \dive \gamma_{\ell+1}
    \overset{\eqref{eq:measure_transp}}&{=}
    \dive \gamma_{\ell} + \frac{1}{2^{(\ell+1)h}} \sum_{P \in \Pc_\ell(Q)} \sum_{T \in \Pc_1(P)} \Big( \delta_{\left(1-2^{-\ell-1}, q(T)\right)} - \delta_{\left(1-2^{-\ell}, q(P) \right)} \Big)
    \\
    \overset{\eqref{eq:diramation_inductive}}&{=}
    -\delta_0 +\frac{2^h}{2^{(\ell+1)h}}\sum_{P \in \Pc_\ell(Q)}\delta_{\left(1-2^{-\ell}, q(P) \right)}
    -\frac{2^h}{2^{(\ell+1)h}} \sum_{P \in \Pc_\ell(Q)}\delta_{\left(1-2^{-\ell}, q(P) \right)} 
    \\
    & \qquad + \frac{1}{2^{(\ell+1)h}} \sum_{T \in \Pc_{\ell+1}(Q)}  \delta_{\left(1-2^{-\ell-1}, q(T)\right)}
     \\
    & =
     -\delta_0 + \frac{1}{2^{(\ell+1)h}} \sum_{T \in \Pc_{\ell+1}(Q)}  \delta_{\left(1-2^{-\ell-1}, q(T)\right)}.
    \end{split}
    \end{equation}
\end{itemize}
Notice that, by definition of $\gamma_{\ell+1}$, it holds
\begin{align*}
	|\gamma_{\ell+1}-\gamma_\ell|(\R^{h+1})
    \overset{\eqref{eq:measure_transp}}&{=}
    \frac{1}{2^{(\ell+1)h}}  \sum_{P\in\Pc_\ell(Q)} \sum_{T \in \Pc_1(P)} \left| \left(1-2^{-\ell}, q(P) \right) - \left(1-2^{-\ell-1}, q(T)\right)  \right|
    \\
    & =
    \frac{\sqrt{h+1}}{2^{\ell+1}}.
\end{align*}
Therefore $\{\gamma_\ell\}_{\ell\in\N}$ is a Cauchy sequence in the strong topology of $\M((-1,1)^{h+1},\R^{h+1})$, hence it strongly converges to some $\gamma\in\M^1((-1,1)^{h+1},\R^{h+1})$. Moreover, since
\begin{equation*}
	\frac{1}{2^{(\ell+1)h}} \sum_{T \in \Pc_{\ell+1}(Q)}  \delta_{\left(1-2^{-\ell-1}, q(T)\right)}
        \weakly  \frac{1}{\Hc{h}( \face{1+})}\Hc{h} \rest \face{1+},
\end{equation*}
\eqref{eq:claim_dive_gamma} holds true.
\end{proof}

We can now pass to the proof of part \ref{item:main_positive} of Theorem \ref{thm:main_scalar}.

\begin{proof}[Proof of part \ref{item:main_positive} of Theorem \ref{thm:main_scalar}]
	Let $\Ac=(A)$, and let $r\coloneqq\rank A\ge1$.
	The case $r=1$ (hence $k=d$) is trivial, since one can choose $\sigma= -f\Lc^d$. Therefore, we assume $r\ge2$.
	Without loss of generality, we may also assume that
	\begin{equation}\label{eq:A=I_rmd}
		A = \Id_{r,d,m}\coloneqq
		\begin{pmatrix}
			\Id_r	& 	0\\
			0		&	0
		\end{pmatrix}\in\R^{d\times m}.
	\end{equation}
	Indeed, one can always find (for instance, using the singular value decomposition of $A$) two invertible matrices $U\in\R^{d\times d}$ and $V\in\R^{m\times m}$ such that
 \begin{equation}
     A = U\Id_{r,d,m}V^T.
 \end{equation}
 Then, letting $\tilde f(x) = (\det U)^{-1}V^Tf(Ux)$, the following implication holds:
 \begin{equation}
     (\Id_{r,d,m})(\tilde f+\tilde\sigma)=0
     \quad
     \implies
     \quad
     \Ac(f+\sigma)=0,
 \end{equation}
 where
 \begin{equation}
     \tilde{\sigma} \coloneqq (V^{-1})^T(u_{\#}\sigma)
     \quad \text{and} \quad
     u(x) \coloneqq Ux.
 \end{equation}
	In the following, we denote by $\{f_1,\dots,f_m\}$ the standard orthonormal basis of $\R^{m}$ and by $\{e_1,\dots,e_d\}$ the standard orthonormal basis of $\R^{d}$. We also let $Q \coloneqq (-1,1)^d$ be the cube in $\R^d$ with faces parallel to $e_1,\dots,e_d$.
	By Proposition \ref{prop:reduction_cubes_basis}, it is sufficient to prove that, for any $j \in \{1,\dots,m\}$, there exists $\sigma \in \M^k(\R^d,\R^m)$ such that
\begin{equation}\label{eq:existence_sigma}
\Ac(f_j \ind{Q} + \sigma)=0.
\end{equation}

\begin{figure}
    \centering
    \includegraphics{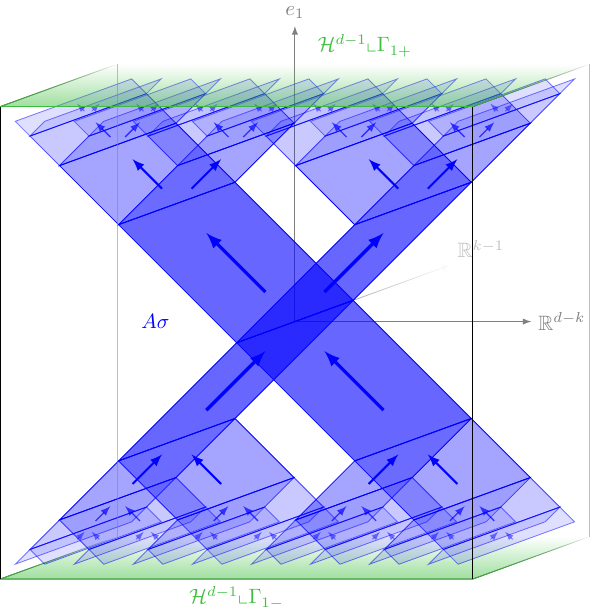}
    \caption{The measure $A \sigma$.}
    \label{fig:sigma_rect}
\end{figure}

It suffices to consider the case $j=1$, since the cases $j\in\{2,\dots,r\}$ are analogous and the cases $j\in\{r+1,\dots,m\}$ are trivial, as $\sigma=0$ satisfies \eqref{eq:existence_sigma}.
We remark that, since $D\ind{Q} = \nu_{\de Q}\Hc{d-1}\rest\de Q$, in the case $j=1$ \eqref{eq:existence_sigma} is equivalent to 
\begin{equation}\label{eq:sigma_balances}
	\Ac\sigma = \Hc{d-1}\rest \face{1+} - \Hc{d-1}\rest \face{1-},
\end{equation}
where $\Gamma_{1+}$ and $\Gamma_{1-}$ are the two faces of $Q$ whose exterior unit normal vectors are $e_1$ and $-e_1$ respectively, as defined in Section \ref{subsec:notations}.

In the rest of the proof, we use the following notation:
\begin{equation}
x=(x',x'') \in Q,
\qquad
\text{where}
\quad
x' \in Q' \coloneqq (-1,1)^{d+1-k},
\quad
x'' \in Q'' \coloneqq (-1,1)^{k-1}.
\end{equation}

We build $\sigma$ as follows.
By Lemma \ref{lmm:1_rect_dive}, there exists a $1$-rectifiable measure $\mu'=\eta'|\mu'|\in \M^1(Q',\R^{d+1-k})$ satisfying
\begin{equation}\label{eq:dive_mu'}
\dive \mu' = \Hc{d-k}\rest \face{1+}' - \Hc{d-k}\rest \face{1-}',
\end{equation}
where $\face{1+}',\face{1-}'$ are the faces of $Q'$ whose exterior unit normal vectors are $e_1$ and $-e_1$ respectively.
Notice that by assumptions $m \geq r \geq d+1-k$, thus the vector
\begin{equation*}
	\eta(x) \coloneqq \big(\eta'(x'),0_{m-(d+1-k)}\big)\in\R^m
\end{equation*}
is well-defined at $|\mu'|\times\Hc{k-1}$-almost every $x$. We may therefore define the measure $\sigma$ by
\begin{equation}\label{eq:def_mu_divergence_main}
\sigma \coloneqq \eta(x) \left( |\mu'| \times \Hc{k-1}\rest Q'' \right).
\end{equation}

Notice that $\sigma \in \M^k(Q,\R^m)$ by construction, since it is the product of the $1$-rectifiable measure $\mu'$ on $Q'$ with the $(k-1)$-rectifiable measure $\Hc{k-1}\rest Q''$. Moreover, by Lemma \ref{lmm:1_rect_dive}, $|\mu'|(Q')$ depends only on $d-k$, hence there exists $C_2>0$ depending only on $d$ such that
\begin{equation}\label{eq:estimate_mu_d}
|\sigma|(Q) \leq C_2
\end{equation}
for any choice of $k$. Since $A\eta=(\eta'(x),0_{k-1})$, it holds
\begin{equation}\label{eq:A_sigma}
    A \sigma =
    (\eta'(x),0_{k-1})\left( |\mu'| \times \Hc{k-1}\rest Q'' \right)
\end{equation}
as represented in Figure \ref{fig:sigma_rect}.

We claim that $\sigma$ satisfies \eqref{eq:sigma_balances}, which would conclude the proof.
In order to do so, let us consider a test function $\phi \in \cinf{d}{}$. We have
\begin{align}\label{eq:dive_mu_main}
\iprod{\Ac\sigma}{\phi}
&=
- \iprod{\sigma}{A^T\nabla \phi}
\\
\overset{\eqref{eq:def_mu_divergence_main}}&{=}
- \int_{Q''} \left ( \int_{Q'} A\eta(x',x'') \cdot \nabla\phi(x',x'') \dif |\mu'|(x') \right ) \dif \Hc{k-1}(x'')
\\
\overset{\eqref{eq:A_sigma}}&{=}
- \int_{Q''} \left ( \int_{Q'} \eta'(x') \cdot \nabla' \phi(x',x'') \dif |\mu'|(x') \right ) \dif \Hc{k-1}(x'')
\\
&= 
\int_{Q''} \iprod{\dive' \mu'}{\phi(\cdot,x'')} \dif \Hc{k-1}(x'')
\\
\overset{\eqref{eq:dive_mu'}}&{=}
\int_{Q''} \left ( \int_{\face{1+}'} \phi(x',x'') \dif \Hc{d-k}(x') - \int_{\face{1-}'} \phi(x',x'') \dif \Hc{d-k}(x')  \right ) \dif \Hc{k-1}(x'')
\\
&= 
\int_{\face{1+}} \phi(x) \dif \Hc{d-1}(x)
- \int_{\face{1-}} \phi(x) \dif \Hc{d-1}(x),
\end{align}
as claimed.
\end{proof}

\subsection{Proof of part \ref{item:main_sharpness} of Theorem \ref{thm:main_scalar} }

For part \ref{item:main_sharpness} of Theorem \ref{thm:main_scalar}, we are going to need two auxiliary results.
The first one may be thought of
as a ``coarea inequality" for Hausdorff measures. The same proposition, in the case $d=2$, is stated and proved in \citep[Proposition 7.9]{falconer2004fractal}. It is easily seen that the same argument shows the following generalization. 

\begin{lemma}\label{lmm:fubini_geq_hausdorff}
Let $k \in \{1,\dots,d\}$, let $s \in [k,d]$ and let $F \subseteq \R^d$ be a Borel set. Then
\begin{equation}
\Hc{s}(F)
\geq
\int_{\R^k} \Hc{s-k}(F \cap T_{x''}) \dif \Hc{k}(x''),
\end{equation}
where $T_{x''} \coloneqq \{x=(x',x'') \in \R^d \colon x' \in \R^{d-k}\}$ for every $x'' \in \R^k$.
\end{lemma}

The second auxiliary result we are going to need is the following, taken from \cite{arroyo2019dimensional}.
\begin{lemma}\label{lemma:aux_divergence}
    Let $\sigma\in\M(\R^{d},\R^d)$ be such that $\dive(\sigma)=0$. Then
    \begin{equation}
        \sigma\ll\haus^{1}.
    \end{equation}
\end{lemma}
\begin{proof}
    We recall the definition of $\ell$-dimensional wave cone $\Lambda^\ell_\Ac$ of a linear differential operator $\Ac \in \op{d,m,n}$, given in \cite{arroyo2019dimensional}:
    \begin{equation}
        \Lambda^\ell_\Ac
        =
        \bigcap_{\pi \in \operatorname{Gr}(\ell,d)}
        \,\,
        \bigcup_{\xi \in \pi \setminus \{0\}} \ker \mathbb{A}(\xi),
    \end{equation}
    where $\operatorname{Gr}(\ell,d)$ is the Grassmaniann of $\ell$-dimensional subspaces of $\mathbb{R}^d$ and $\mathbb{A}$ is the Fourier symbol of the operator $\mathcal{A}$, which we define later in \eqref{eq:Fourier_symbol}.
    
    In \cite[Corollary 1.4]{arroyo2019dimensional}, it was proved that, if the $\ell$-wave cone $\Lambda^\ell_\Ac$ of a linear differential operator $\Ac$ satisfies $\Lambda^\ell_\Ac=\{0\}$, then any Radon measure $\sigma$ such that $\Ac\sigma=0$ satisfies $\sigma\ll \haus^\ell$.
    It is easy to check that, if $\Ac=(A)\in\op{d,m,1}$, then (some further observations on this matter are given later in Subsection \ref{subsec:wavecone})
    \begin{equation}
        \Lambda^1_\Ac = \{x\in\R^m\colon Ax\cdot\xi=0\mbox{ for all }\xi\in\mathbb{S}^{d-1}\}.
    \end{equation}
    In particular, since $\dive = (\Id)$, it holds $\Lambda^1_{\dive}=\{0\}$, hence the desired result.
\end{proof}

\begin{proof}[Proof of part \ref{item:main_sharpness} of Theorem \ref{thm:main_scalar}]
By the same argument used in the proof of part \ref{item:main_positive} of Theorem \ref{thm:main_scalar}, we may assume that
\begin{equation}
A =
\begin{pmatrix}
\Id_r	& 	0\\
0		&	0
\end{pmatrix}\in\R^{d\times m},
\end{equation}
where $\Id_r$ denotes the $r$-dimensional identity matrix. We also let $\{f_1,\dots,f_m\}$ and $\{e_1,\dots,e_d\}$ denote the canonical orthonormal basis of $\R^{m}$ and $\R^d$, respectively. We are going to prove that we can take $f=f_1 \ind{Q}$.

Let us fix $\sigma= \eta |\sigma| \in \M(\R^d,\R^m)$ satisfying
\begin{equation}
\Ac(f_1 \ind{Q} + \sigma)=0,
\end{equation}
where $Q=(-1,1)^d$.
We are going to show that $|\sigma| \not \perp \Hc{d+1-r}$, that is
\begin{equation}
\Hc{d+1-r}(E) \neq 0
\qquad
\mbox{for every } E \subseteq \R^d \colon |\sigma|(E^c)=0.
\end{equation}

As in the first part of the proof, by Lemma \ref{lmm:Amu_dive} the equation above means that
\begin{equation}\label{eq:dive_sigma_hd-1}
\dive(A \sigma)=
\Hc{d-1}\rest \face{1+} - \Hc{d-1}\rest \face{1-}.
\end{equation}
\eqref{eq:dive_sigma_hd-1} in particular implies $\sigma \neq 0$. In the rest of the proof, we use the notations
\begin{equation}
\begin{gathered}
x =(x',x'') \in \R^d
\qquad
\text{where}
\qquad
x' \in \R^r \quad \text{and} \quad x'' \in \R^{d-r},
\\[4pt]
\face{1\pm}' \coloneqq \{\pm1\} \times (-1,1)^{r-1},
\qquad
Q'' \coloneqq (-1,1)^{d-r}
\end{gathered}
\end{equation}
so that $\Gamma_{1\pm} = \Gamma'_{1\pm}\times Q''$.
By disintegrating $|\sigma|$ with respect to $x''$ (see, for instance, \citep[Theorem 5.3.1]{AmbrosioGigliSavare}) we can write 
\begin{equation}
\int \phi(x) \dif |\sigma|(x)
=\int_{\R^{d-r}} \left ( \int_{\R^r} \phi(x',x'') \cdot \big(\eta'(x',x''),\eta''(x',x'') \big) \dif \sigma_{x''}(x') \right ) \dif \lambda(x'')
\qquad
\forall
\phi \in C_c(\R^d,\R^d),
\end{equation}
where
\begin{equation}
\eta'(x',x'') \in \R^r,
\qquad
\eta''(x',x'') \in \R^{d-r},
\qquad
\lambda \in \M(\R^{d-r},[0,+\infty]),
\qquad
\sigma_{x''}(\R^r)=1
\quad
\text{for $\lambda$-a.e. $x'' \in \R^{d-r}$}.
\end{equation}
We now test \eqref{eq:dive_sigma_hd-1} with $\zeta(x)\coloneqq \phi(x')\psi(x'')$, where $\phi \in \ccinf{r}{}$ and $\psi \in \ccinf{d-r}{}$: taking into account that
\begin{equation}
\nabla \zeta(x)
=
\big( \nabla' \phi,0''\big)\psi(x'') + \big(0',\nabla'' \psi \big) \phi(x'),
\end{equation}
we obtain
\begin{equation}\label{eq:test_part_2_main}
\begin{split}
\int_{Q''} \psi(x'') \left (\int_{\face{1+}'} \phi(x') \dif \Hc{r-1} - \int_{\face{1-}'} \phi(x') \dif \Hc{r-1}  \right )& \dif \Hc{d-r}(x'')
=
\int_{\face{1+}} \phi \psi \dif \Hc{d-1}
-
\int_{\face{1-}} \phi \psi \dif \Hc{d-1}
\\
= &
\int A\eta \cdot \nabla \zeta \dif |\sigma|
\\
= &
\int_{\R^{d-r}} \psi(x'') \left ( \int_{\R^r}  \eta'(x',x'')\cdot \nabla' \phi(x') \dif \sigma_{x''}(x') \right ) \dif \lambda(x'').
\end{split}
\end{equation}
Choosing ${\bar{\phi}}$ such that ${\bar{\phi}}_{|\face{+1}'}\equiv 2^{-r+1}$ and ${\bar{\phi}}_{|\face{1-}'}\equiv 0$, the arbitrariness of $\psi$ in \eqref{eq:test_part_2_main} implies
\begin{equation}\label{eq:Hc_d-r_ll_lambda}
\Hc{d-r} \rest Q''
=
\beta \lambda,
\end{equation}
where $\beta(x'')\coloneqq \int_{\R^r}  \eta'(x',x'')\cdot \nabla' \bar{\phi}(x') \dif \sigma_{x''}(x')$. \eqref{eq:Hc_d-r_ll_lambda} yields
\begin{equation}\label{eq:beta_positive}
\beta(x'') > 0
\qquad
\text{for $\Hc{d-r}$-a.e. $x'' \in Q''$}.
\end{equation}
Next, inserting \eqref{eq:Hc_d-r_ll_lambda} in \eqref{eq:test_part_2_main}
and using the arbitrariness of $\psi$ again we deduce that, for any choice of $\phi$, it holds
\begin{equation}
\beta(x'') \left (\int_{\face{1+}'} \phi(x') \dif \Hc{r-1} - \int_{\face{1-}'} \phi(x') \dif \Hc{r-1}
\right )
=
\int_{\R^r}  \eta'(x',x'')\cdot \nabla \phi(x') \dif \sigma_{x''}(x')
\qquad
\text{for  $\lambda$-a.e. $x'' \in \R^{d-r}$}.
\end{equation}
The set of full $\lambda$-measure where the above equality holds may in principle depend on the choice of $\phi$. Nevertheless, since $C_c^1(\R^d,\R)$ is separable, fixing a countable dense subset $X$, there exists a set of full $\lambda$-measure where the equality is true for any $\phi \in X$. By density of $X$, on the same set the equality holds for every $\phi \in C_c^1(\R^d,\R)$, hence
\begin{equation}\label{eq:dive_r_sigma}
\dive' \big( \eta'\sigma_{x''} \big)
=
\beta(x'') \left ( \Hc{r-1} \rest \face{1+}' - \Hc{r-1} \rest \face{1-}' \right )
\qquad
\text{for $\lambda$-a.e. $x'' \in \R^{d-r}$}.
\end{equation}
We can read this relation as
\begin{equation}
\dive' \left (
\eta'(\cdot,x'')\sigma_{x''} + {\beta(x'')} e_1'\ind{Q'}
\right )
=
0
\qquad
\text{for $\lambda$-a.e. $x'' \in \R^{d-r}$},
\end{equation}
where $e_1'=(1,0,\dots,0) \in \R^r$ and $Q'=(-1,1)^r \subset \R^r$.
By Lemma \ref{lemma:aux_divergence}, it holds
\begin{equation}\label{eq:sigma_x''_ll_H1}
\sigma_{x''} \ll \Hc{1}
\qquad
\text{for $\lambda$-a.e. $x'' \in \R^{d-r}$}.
\end{equation}

Let us now choose $E \subseteq \R^d$ such that $|\sigma|(\R^d \setminus E)=0$; we are going to show that $\Hc{d-r+1}(E) \neq 0$, proving that $|\sigma| \not \perp \Hc{d-r+1}$.

Let us fix any $F'' \subseteq Q''$ with $\lambda(F'')>0$ and call $F \coloneqq \R^r \times F''$. Since $\lambda(F'')>0$, by \eqref{eq:beta_positive} we have
\begin{equation}\label{eq:Hc_E''_positive}
\Hc{d-r}(F'')>0
\end{equation}
as well.
Letting $T_{x''} \coloneqq \{x=(x',x'') \in \R^d \colon x' \in \R^{r}\}$, it holds
\begin{equation}
0
<
\lambda(F'')
=
|\sigma|(F)
=
|\sigma|(E \cap F)
=
\int_{F''}  \sigma_{x''}(E \cap T_{x''}) \dif \lambda(x'')
\leq
\lambda(F''),
\end{equation}
where the first equality and the last inequality follow by the fact that $\sigma_{x''}$ are probability measures and the second equality is a consequence of $|\sigma|(\R^d \setminus E)=0$. The above relation yields
\begin{equation}
\sigma_{x''}(E \cap T_{x''})=1
\qquad
\text{for $\lambda$-a.e. $x'' \in F''$}
\end{equation}
which, by \eqref{eq:sigma_x''_ll_H1}, produces
\begin{equation}
\Hc{1}(E \cap T_{x''})>0
\qquad
\text{for $\lambda$-a.e. $x'' \in F''$}.
\end{equation}
Hence, by \eqref{eq:Hc_d-r_ll_lambda} and \eqref{eq:beta_positive}, we infer
\begin{equation}\label{eq:H1_A_positive}
\Hc{1}(E \cap T_{x''})>0
\qquad
\text{for $\Hc{d-r}$-a.e. $x'' \in F''$}.
\end{equation}
Thus we have
\begin{equation}
\Hc{d-r+1}(E)
\geq
\Hc{d-r+1}(E \cap F)
\geq
\int_{F''} \Hc{1}(E \cap T_{x''}) \dif \Hc{d-r}(x'')
>
0,
\end{equation}
where the second inequality follows by Lemma \ref{lmm:fubini_geq_hausdorff} with $k=d-r$ and $s=d-r+1$, while the last inequality follows by \eqref{eq:Hc_E''_positive} and \eqref{eq:H1_A_positive}.
\end{proof}

We conclude this section with an example of an $\Ac$-free measure which is not absolutely continuous with respect to $\Hc{d+1-\rank A}$, thus proving the assertion of Remark \ref{rem:sharp_sharpness}.

\begin{example}\label{ex:sharp_sharpness}
    Let us fix $2 \leq r \leq d-1$ and $\Ac=(A)$, where
    \begin{equation}
     A=
    \begin{pmatrix}
      \Id_r & 0
      \\
      0     & 0
    \end{pmatrix}
    \in \R^{d \times m},
    \end{equation}
    and let $\mu=\gamma' \Hc{1} \rest \gamma$, where $\gamma$ is a parametrization of a closed regular curve lying in $\spn(e_1,e_2)$. It holds
    \begin{equation}
        \Ac \mu = \dive (A \mu) = \dive \mu = 0,
    \end{equation}
    although $|\mu| \not \ll \Hc{d+1-r}$, since $d+1-r \geq 2$. 
\end{example}

\section{A sufficient condition for balancing vector-valued operators}\label{sec:vector_valued}
In this section we state Condition \ref{hyp:solving_2} and we prove Theorem \ref{thm:main_vector_valued}. We refer the reader to Section \ref{subsec:notations} for the relevant notation.

\begin{hypo}\label{hyp:solving_2}
There exists a basis $\mathcal{F}$ of $\R^m$ such that, for every $f \in \mathcal{F}$, there exist $\Bc=(B^1,\dots,B^n) \cong \Ac$, $g_1,\dots,g_n \in \R^m$, $e \in \mathbb{S}^{d-1}$ and $p \in \{0,\dots,d\}$ with the following properties.
\begin{enumerate}
\item\label{pnt:hyp_suff_1} $(B^1f,\dots,B^{n}f) = (e_1,e_2,\dots,e_{p},0,\dots,0)$, where $(e_1,e_2,\dots,e_{p})$ is an orthonormal system in $\R^d$.
\item\label{pnt:hyp_suff_2} The following relations are satisfied:
\noeqref{eq:hyp_suff_2a,eq:hyp_suff_2b,eq:hyp_suff_2c}
\begin{subequations}\label{eq:hyp_suff_2}
\begin{gather}
B^kg_k=e
\quad
\forall k \in \{1,\dots,p\};\label{eq:hyp_suff_2a}
\\[5pt]
B^\l g_k \in \spn(e_1,\dots,e_p)
\qquad
\forall k,\l \in \{1,\dots, n\}, \l \neq k;\label{eq:hyp_suff_2b}
\\[5pt]
B^\l g_k \cdot B^h f
=
- B^\l g_h \cdot B^k f
\qquad
\forall \l,h,k \in \{1,\dots,n\}.\label{eq:hyp_suff_2c}
\end{gather}
\end{subequations}
\end{enumerate}
\end{hypo}
\begin{remark}
It clearly follows $p \leq n$ as well. Moreover, the conditions \eqref{eq:hyp_suff_2} are not restrictive for $h,k \in \{p+1,\dots,n\}$, since one can choose $g_k=0$ for such $k$.
\end{remark}

Since $\Ac \cong \Bc$ yields $\Ac \mu=0$ if and only if $\Bc \mu=0$ and by virtue of Proposition \ref{prop:reduction_cubes_basis}, the following statement implies Theorem \ref{thm:main_vector_valued}.

\begin{proposition}\label{prop:hyp_sufficient_solving}
Let $\Ac$, $f \in \mathcal{F}$ and $\Bc$ be as in Condition \ref{hyp:solving_2} and let $B=B_1(0) \subset \R^d$. Then there exists $g \in C^1(\de B, \R^m)$ such that
\begin{equation}\label{eq:Bc_g=balance}
\Bc\big(f \ind{B} + g \hu \rest \de B \big)
=
0.
\end{equation}
\end{proposition}



\begin{remark}\label{rem:intuition_condition}
Condition \ref{hyp:solving_2} above is motivated by the following considerations.
Let us fix $f \in \mathbb{R}^m$ and an operator $\mathcal{B} \cong \mathcal{A}$ satisfying part \ref{pnt:hyp_suff_1} of Condition \ref{hyp:solving_2}.
For each component $B^k$ of $\mathcal{B}$ which ``sees" the function $f \ind{B}$, namely $B^1,\dots,B^p$, it is relatively easy to construct a balancing vector field $h_k$ on $\partial B$.
The idea behind Condition \ref{hyp:solving_2} is to choose $h_k$ so that the sum $g=h_1+\dots +h_p$ to balances $f \ind{B}$ for all components of $\mathcal{B}$.

As already stated in the introduction, the existence of $g_1,\dots,g_p$ satisfying \eqref{eq:hyp_suff_2a} is a sort of ``shared rank-2 property" of the matrices $B^1,\dots,B^p$ and allows the presence of $f$ in the definition of $h_k$ coherently with the other components of $\mathcal{B}$. In fact $h_k=x_{p+1}x_k g_k + x_{p+1}^2f$ balances $f \ind{B}$ for $B^k$.

\eqref{eq:hyp_suff_2b} is a technical condition which ensures $B^\ell g \in \operatorname{span}(e_1,\dots,e_{p+1})$ and
\eqref{eq:hyp_suff_2c}, as already stated in the introduction, is a sort of antisymmetry of the tensor defining $\mathcal{B}$. Both ensure that the vector field $B^\ell h_k$ on $\partial B$ is $\mathcal{B}$-free for $\ell \neq k$.
\end{remark}

The remaining part of the section is devoted to the proof of Proposition \ref{prop:hyp_sufficient_solving}.
We begin with the following algebraic consequences of Condition \ref{hyp:solving_2}.

\begin{remark}[consequences of Condition \ref{hyp:solving_2}]
    Taking $h=k$ in \eqref{eq:hyp_suff_2c}, one obtains
    \begin{equation}\label{eq:cons_1_hyp_solving_2}
B^\l g_k \perp B^k f
\qquad
\forall \l,k \in \{1,\dots,n\}.
\end{equation}
In particular, choosing $k=\ell$ in the above equation gives
\begin{equation}\label{eq:e_perp_Bf}
e \perp B^k f
\qquad
\forall k \in \{1,\dots, n\}.
\end{equation}
We moreover record
\begin{equation}\label{eq:cons_2_hyp_solving_2}
B^k g_h \perp e_k
\qquad
\forall k \in \{1,\dots,p\},\,\,\forall h \in \{1,\dots,n\}.
\end{equation}
Indeed, \eqref{eq:hyp_suff_2c} with $\l=k \in \{1,\dots,p\}$ yields
\begin{equation}
0
\overset{\eqref{eq:e_perp_Bf}}{=}
e
\cdot B^h f
\overset{\eqref{eq:hyp_suff_2a}}{=}
B^k g_k \cdot B^h f
\overset{\eqref{eq:hyp_suff_2c}}{=}
- B^k g_h \cdot B^k f
\overset{(\ref{pnt:hyp_suff_1})}{=}
- B^k g_h \cdot e_k
\qquad
\forall h \in \{1,\dots,n\}.
\end{equation}
\end{remark}

The fact that the measure $\sigma = g \hu \rest \de B$ whose existence is claimed by Proposition \ref{prop:hyp_sufficient_solving} is defined on a smooth surface implies the following characterization of being $\Ac$-free, which we use in the proof of Proposition \ref{prop:hyp_sufficient_solving}.

\begin{lemma}\label{lmm:equivalent_Afree}
Let $f \in \R^m$ and let $B=B_1(0) \subset \R^d$.
Then a vector field $g \in L^1_{\hu}(\de B, \R^m)$ satisfies
\begin{equation} \label{eq:balanced_afree}
\Ac \big(f \ind{B}+g \hu \rest \de B\big)
=
0
\end{equation}
if and only if, for every $k \in \{1,\dots,n\},$ the following conditions hold:
\begin{subequations}\label{eq:diveAg_equals_Af_nu}
\begin{gather}
A^k g \perp \nu_{\de B}
\qquad
\text{on } \de B;
\label{eq:Ag_perp_face}
\\[5pt]
\dive_{\de B} A^kg
=
- (A^kf) \cdot \nu_{\de B}
\qquad
\Hc{d-1}\text{a.e.\ on } {\de B}.
\label{eq:dive_Ag_f_normal}
\end{gather}
\end{subequations}
\end{lemma}

\begin{proof}
   Since we can argue component-wise on $\Ac$, we assume that $\Ac=(A)$ for some $A \in \mathbb{R}^{d \times m}$.
   \begin{description}
        \item[$\bullet$ Step 1: if $\Ac \big(g \hu \rest \de B\big)$ is a measure, then \eqref{eq:Ag_perp_face} holds]
       Let $\mu \coloneqq \Ac \big(g \hu \rest \de B\big)$. By Lemma \ref{lmm:Amu_dive}, we have
\begin{equation}\label{eq:dive_is_measure}
	\mu = \dive(Ag\hu\rest \de B).
\end{equation}
Notice that $\supp\mu\subset\supp (g\hu\rest\de B)\subset\de B$.
Given $\phi\in C^\infty(\de B,\R)$  and $\gamma\in C^\infty_c((0,2),\R)$ such that $\gamma\equiv1$ in a neighborhood of $1$, let
\begin{equation*}
	\psi(x) = (|x|-1) \gamma(|x|)\phi\left( \frac{x}{|x|} \right).
\end{equation*}
Since $0 \notin \supp \gamma$, $\psi\in C^\infty_c(\R^d)$ and it holds
\begin{equation*}
	\nabla\psi(x) = \frac{x}{|x|}\gamma(|x|)\phi\left( \frac{x}{|x|} \right) + (|x|-1)\gamma'(|x|)\frac{x}{|x|}\phi\left( \frac{x}{|x|} \right) + (|x|-1)\gamma(|x|) \nabla_{\de B}\phi\left( \frac{x}{|x|} \right) \frac{1}{|x|}.
\end{equation*}
On the other hand, $\psi\equiv0$ on $\de B$. Therefore, by \eqref{eq:dive_is_measure}, we have
\begin{equation*}
	0=\iprod{\mu}{\psi} = \int_{\de B} Ag(x)\cdot\nabla\psi(x)\dif\hu(x)
 = \int_{\de B} Ag(x) \cdot \frac{x}{|x|}\phi\left( \frac{x}{|x|} \right)\dif\hu(x).
\end{equation*}
Since $\phi$ is arbitrary, we deduce $Ag(x) \cdot \frac{x}{|x|}=0$ for $\hu$-a.e. $x\in\de B$.
    \item[$\bullet$ Step 2: given \eqref{eq:Ag_perp_face}, \eqref{eq:balanced_afree} and \eqref{eq:dive_Ag_f_normal} are equivalent]
    Since \eqref{eq:Ag_perp_face} necessarily holds, to finish the proof it is sufficient to assume  $Ag(x) \perp \nu_{\de B}(x)$ for $\hu$-a.e. $x \in \de B$ and prove that \eqref{eq:balanced_afree} and \eqref{eq:dive_Ag_f_normal} are equivalent. We have
\begin{equation}
\Ac(f \ind{B})
=
- \scal{Af}{\nu_{\de B}} \hu \rest \de B.
\end{equation}
Hence, by Lemma \ref{lmm:Amu_dive}, \eqref{eq:balanced_afree} is equivalent to
\begin{equation}\label{eq:equivalent_dive_faces}
-\scal{Af}{\nu_{\de B}} \hu \rest \de B
=
\dive(A g \hu \rest \de B).
\end{equation}
On the other hand we have
\begin{equation}
   \begin{split}
    \dive(A g \hu \rest \de B)
    = &
    \dive_{\de B}(A g \hu \rest \de B) + \de_{\nu_{\de B}} \big(Ag \cdot \nu_{\de B} \hu \rest \de B  \big)
    \\
    \overset{\eqref{eq:Ag_perp_face}}&{=}
    \dive_{\de B}(A g) \hu \rest \de B,
   \end{split} 
\end{equation}
obtaining the conclusion.\qedhere
   \end{description}
\end{proof}

\begin{proof}[Proof of Proposition \ref{prop:hyp_sufficient_solving}]
Since by \eqref{eq:e_perp_Bf} we have $e\perp e_k$ for $k \in \{1,\dots,p\}$, we
call $e_{p+1}\coloneqq e$ and (if $p+1<d$) we choose $\{e_{p+2},\dots, e_{d}\}$ so that $\{e_1,\dots,e_d\}$ is an orthonormal basis of $\R^d$; up to a change of coordinates, we may assume $\{e_1,\dots,e_d\}$ is the standard orthonormal basis of $\R^d$.

We define $g \colon \de B \to \R^m$ as
\begin{equation}
    g(x)
    =
    -\left( x_{p+1}\sum_{h=1}^p x_h g_h \right) +x_{p+1}^2 f 
    \qquad
    \forall x \in \de B.
\end{equation}
where $g_1,\dots,g_p$ are given by Condition \ref{hyp:solving_2}. $B^\ell g$ is illustrated in Figure \ref{fig:vector}.
\begin{figure}
    \centering
    \includegraphics[scale=0.95]{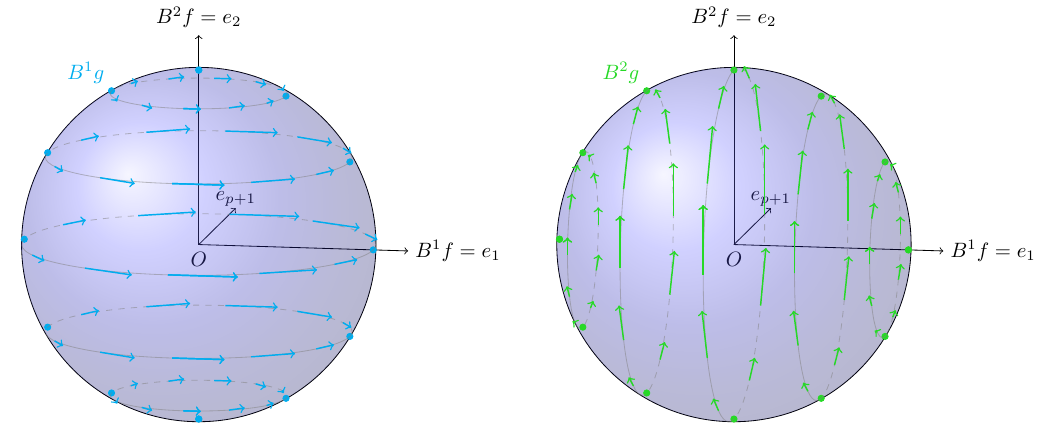}
    \caption{illustration of $B^\ell g$ for $d=3$ and $n=p=2$. Condition \ref{hyp:solving_2} implies that $B^1 f \perp B^2 f$ and both these vectors are perpendicular to $e=e_{p+1}$. Notice that, since $d=3$, Condition \ref{hyp:solving_2} implies $B^\ell g_h=0$ for $\ell \neq k$, thus $B^\ell g \in \spn(e_\ell,e_{p+1})$.}
    \label{fig:vector}
\end{figure}

By Lemma \ref{lmm:equivalent_Afree} it suffices to show that $g$ satisfies \eqref{eq:Ag_perp_face} and \eqref{eq:dive_Ag_f_normal} for every $k \in \{1,\dots,n\}$.

\begin{enumerate}[label=\textit{\alph*)}]
    \item Since $\nu_{\de B}(x)=x$, for every  $\ell \in \{1,\dots,n\}$ we have to show that
    \begin{equation}
        B^\ell g(x) \cdot x=0
        \qquad
        \forall x \in \de B.
    \end{equation}
    If $\ell \in \{1,\dots,p\}$, then
        \eqref{eq:hyp_suff_2a} and \eqref{eq:hyp_suff_2b} yield $B^\ell g(x) \in \spn(e_1,\dots,e_{p+1}) $, hence
        \begin{equation}\label{eq:vector_piu_dettagli}
            \begin{split}
                B^\ell g(x) \cdot x
                & =
                -x_{p+1} \left( \sum_{k=1}^p x_k B^\ell g_k - x_{p+1}B^\ell f\right) \cdot \left( \sum_{h=1}^{p+1} x_h e_h \right)
                \\
                & =
                -x_{p+1} \left( \sum_{k=1}^p x_k B^\ell g_k - x_{p+1}B^\ell f \right) \cdot \left( \sum_{h=1}^{p} x_h B^h f + x_{p+1}B^\ell g_\ell\right)
                \\
                & =
                - x_{p+1} \sum_{1 \leq k< h \leq p} x_kx_h \overbrace{\left(B^\ell g_k \cdot  B^h f + B^\ell g_h \cdot B^k f \right)}^{=0 \text{ by } \eqref{eq:hyp_suff_2c}}
                -  x_{p+1}^2 \overbrace{\sum_{k=1}^p x_k B^\ell g_k \cdot B^\ell g_\ell}^{=x_\ell \text{ by } \eqref{eq:hyp_suff_2b}, \eqref{eq:hyp_suff_2a}}
                \\
                & \quad
                + x_{p+1}^2  \underbrace{B^\ell f \cdot \sum_{h=1}^{p} x_h B^h f}_{=x_\ell \text{ by pnt.\  (\ref{pnt:hyp_suff_1})}}
                + x_{p+1}^3  \underbrace{B^\ell f \cdot B^\ell g_\ell}_{=0 \text{ by } \eqref{eq:cons_1_hyp_solving_2}}
                \\
                & =
                0.
            \end{split}
        \end{equation}
         
         If $\ell \in \{p+1,\dots,n\}$, then $B^\ell f=0$ and $B^\ell g(x) \in \spn(e_1,\dots,e_p)$ by \eqref{eq:hyp_suff_2b}. Hence
         \begin{equation}
             \begin{split}
                 B^\ell g \cdot x
                 & =
                 x_{p+1} \left( \sum_{k=1}^p x_k B^\ell g_k \right) \cdot \left( \sum_{h=1}^p x_h B^h f\right)
                 \\
                 & =
                 x_{p+1} \sum_{k=1}^p x_k^2 \underbrace{B^\ell g_k \cdot B^k f}_{=0 \text{ by } \eqref{eq:cons_1_hyp_solving_2}}
                 \,\,+
                 x_{p+1} \sum_{1 \leq k< h \leq p} x_kx_h \underbrace{\left(B^\ell g_k \cdot  B^h f + B^\ell g_h \cdot B^k f \right)}_{=0 \text{ by } \eqref{eq:hyp_suff_2c}}
                 \\
                 & =
                 0.
             \end{split}
         \end{equation}

    \item Now we have to prove \eqref{eq:dive_Ag_f_normal}, namely that for every $\ell \in \{1,\dots,n\}$ it holds
    \begin{equation}
        -\dive_{\de B} B^\ell g(x)
        =
        B^\ell f \cdot x
        \qquad
        \forall x \in \de B.
    \end{equation}
    We first assume $\ell \in \{1,\dots,p\}$ and, for simplicity of notations, that $\ell=1$. Hence $B^1 f \cdot x= x_1$ for every $x \in \de B$. We compute
    \begin{equation}
    \begin{split}
        -\dive_{\de B} B^1 g(x)
        & =
        P_{T_x \de B}\nabla x_{p+1} \cdot \Big( x_1 B^1g_1 + x_2 B^1 g_2 + \dots + x_p B^1 g_p - x_{p+1} B^1 f\Big)
        \\
        & \quad +
        x_{p+1}\dive_{\de B}\Big( x_1 B^1g_1 + x_2 B^1 g_2 + \dots + x_p B^1 g_p - x_{p+1} B^1 f\Big)
        \\
        &=
        e_{p+1} \cdot \Big( x_1 e_{p+1} + x_2 B^1 g_2 + \dots + x_p B^1 g_p - x_{p+1} e_1\Big)
        \\
         & \quad +
        x_{p+1} \dive \Big( x_1 e_{p+1} + x_2 B^1 g_2 + \dots + x_p B^1 g_p - x_{p+1} e_1\Big)
        \\
        \overset{\eqref{eq:hyp_suff_2b}}&{=}
        x_1 + x_{p+1} \Big( e_1 \cdot e_{p+1} + e_2 \cdot B^1 g_2 + \dots + e_p \cdot B^1 g_p - e_{p+1} \cdot e_1 \Big)
        \\
        \overset{\eqref{eq:cons_1_hyp_solving_2}}&{=}
        x_1,
    \end{split}
    \end{equation}
    as desired; notice that in the second equality we used the fact that the projection $P_{T_x \partial B}$ is self-adjoint and that the second term of the scalar product lies on $T_x \partial B$, as we proved in part \textit{a)}; in the same equality, the tangential divergence becomes a full divergence since by $x \perp T_x \partial B$ it follows
    \begin{equation}
        \operatorname{div}_{\partial B} B^1 g
        =
        \operatorname{div} B^1 g - (\partial_x B^1 g) \cdot x
        =
        \dive B^1 g,
    \end{equation}
    where last equality is due to $\partial_x B^1 g \in T_x \partial B$, which in turn follows by $B^1 g(x) \cdot x \equiv 0$.
    
     On the other hand, if $\ell \in \{p+1,\dots,n\}$, then $B^\ell f=0$, hence $B^\ell f \cdot x =0$. Moreover, with similar computations as above,
    \begin{equation*}
    \begin{split}
        \dive_{\de B} B^\ell g(x)
        & =
        e_{p+1} \cdot \sum_{k=1}^p x_k B^\ell g_k
        +
        x_{p+1} \dive \sum_{k=1}^p x_k B^\ell g_k
        \\
        \overset{\eqref{eq:hyp_suff_2b}}&{=}
        0 +  x_{p+1} \sum_{k=1}^p e_k \cdot B^\ell g_k
        \\
        \overset{\eqref{eq:cons_1_hyp_solving_2}}&{=}
        0.\qedhere
    \end{split}
    \end{equation*}
\end{enumerate}

\end{proof}

\section{Examples and further observations on Condition \ref{hyp:solving_2}}\label{sec:examples_and_obs}

\subsection{Exterior derivatives}\label{sec:exterior_derivative}

In this Subsection, we prove Corollary \ref{cor:diff_forms}.
Namely, given\footnote{The case $q=0$ is trivial since any $d$-form in $\R^d$ is closed, while the case $q=d$ is the differential acting on scalar functions, which is not $k$-balanceable for any $k \in \{1,\dots,d-1\}$ because it corresponds to the gradient.} $q \in \{1,\dots,d-1\}$, we show that the exterior derivative operator $\Ac$ acting on $(d-q)$-forms satisfies Condition \ref{hyp:solving_2}, thus showing that it is $(d-1)$-balanceable.
In other words, we prove that any $(d-q)$-form $f \in L^1(\R^d)$ is the absolutely continuous part of a \emph{closed} $(d-q)$-form.

\begin{remark}
    By identifying vector fields $\R^d\to\R^d$ with $1$-forms, Corollary \ref{cor:diff_forms} recovers Alberti's original result (\citep[Theorem 3]{alberti1991lusin}), while identifying vector fields with $(d-1)$-forms, Corollary \ref{cor:diff_forms} recovers the fact that $\dive$ is $(d-1)$-balanceable (consistently with Theorem \ref{thm:main_scalar}, see also Subsection \ref{sec:scalar_comparing_vector}).
\end{remark}


We start by fixing some notation.
Throughout the section, we will be working with $(d-p)$-multi-indices (i.e. $(d-p)$-tuples $I=(i_1,\dots,i_{d-p})$ such that $1\le i_1<i_2<\dots<i_{d-p}\le d$), where $1\le p\le d$ is an integer.
We also introduce the standard orthonormal basis $\mathcal{F}$ for the space $\Lambda^{d-p}(\R^d)$ of $(d-p)$-covectors:
\begin{equation}\label{eq:ext_basis}
    \mathcal{F} \coloneqq \big\{f_I\coloneqq \dif x_{i_1}\wedge \dif x_{i_2}\wedge\dots \wedge \dif x_{i_{d-p}} \,\,\,\text{ where $I=(i_1,i_2,\dots,i_{d-p})$ is a $(d-p)$-multi-index}\big\}.
\end{equation}
Let us now fix $q \in \{1,\dots,d-1\}$.
Given a $(d-q)$-multi-index $I$ and a $(d-q+1)$-multi-index $J$:
\begin{itemize}
        \item We write $I \subset J$ if $J$ can be obtained by adding a component to $I$.
        \item If $I\subset J$, we write $J \setminus I$ to denote the only $\{\ell\} \subset \{1,\dots,d\}$ which produces $J$ when added to $I$.
\end{itemize}
Lastly, we fix the standard orthonormal basis $\{e_\ell\}_{\ell=1}^d$ of $\R^d$. 

The exterior derivative acting on $(d-q)$-forms is a first-order linear operator
\begin{equation}
    \Ac \in C^\infty\big(\R^d,\Lambda^{d-q}(\R^d)\big) \to C^\infty\big(\R^d,\Lambda^{d-q+1}(\R^d)\big),
\end{equation}
hence $\mathcal{A} \in \op{d,m,n}$ with
\begin{equation}
    m \coloneqq \dim \Lambda^{d-q}(\R^d)= \binom{d}{d-q},
    \qquad
    n \coloneqq \dim \Lambda^{d-q+1}(\R^d) = \binom{d}{d-q+1}.
\end{equation}
With the notation introduced in Section \ref{subsec:notations}, we may represent the exterior derivative with $n$ matrices $A^J\in \mathbb{R}^{d \times m}$ where $J=(j_1,j_2,\dots,j_{d-q+1})$ ranges over the set $\{J_1,\dots,J_n\}$ of $(d-q+1)$-multi-indices, and
    \begin{equation}\label{eq:ext_def_A_J}
        A^J f_I=
        \begin{cases}
        (-1)^{\ell+1} e_\ell        & \text{if } I=(j_1,j_2,\dots,j_{\ell -1}, j_{\ell+1}, \dots j_{d-q+1})
        \\
        0                   & \text{if } I \not \subset J.
        \end{cases}
    \end{equation}

\begin{proposition}\label{prop:ext_satisfies_hypo}
The exterior derivative $\Ac$ satisfies Condition \ref{hyp:solving_2} with $\Bc=(A^{J_1},\dots,A^{J_n})$ chosen above and $\mathcal{F}$ being the standard basis of $\Lambda^{d-q}(\R^d)$ defined in \eqref{eq:ext_basis}.
\end{proposition}


\begin{proof}[Proof of Proposition \ref{prop:ext_satisfies_hypo}]

   For simplicity of notations we fix $I\coloneqq (q+1,q+2,\dots,d)$ and we prove that items 1 and 2 of Condition \ref{hyp:solving_2} hold true for $f_I=f_{(q+1,q+2,\dots,d)}\in\Fc$.
   The argument is exactly the same for any other $(d-q)$-multi-index.
   Notice that, with this choice, if $I\subset J$ then $J=(\ell,q+1,q+2,\dots,d)\eqqcolon (\ell,I)$ for some $\ell\in\{1,\dots,q\}$.
   
   By \eqref{eq:ext_def_A_J}, it holds
    \begin{equation}\label{eq:ext_def2}
        A^J f_I
        =
        \begin{cases}
            e_\ell              & \text{if } J=(\ell,I) \text{ for some } \ell \in (1,\dots,q)
            \\
            0                   & \text{if } J \not \supset I.
        \end{cases}
    \end{equation}
    We choose
    \begin{equation}\label{eq:ext_choices}
        p=q,
        \qquad
        \{B^\ell\}_{\ell=1}^{p}=\{A^{(\ell,I)}\}_{\ell=1}^{p},
        \qquad
        e=(-1)^{d-p} e_d,
        \qquad
        g_\ell =
        \begin{cases}
            f_{(\ell,p+1,p+2,\dots,d-1)}                & \text{if } \ell \in \{1,\dots,p\}
            \\
            0                                   & \text{if } \ell \in \{p+1,\dots,n\}.
        \end{cases}
    \end{equation}
    We also choose $\{B^\ell\}_{\ell \geq p+1}$ to be some ``listing" of $\{A^J\}_{J \not \supset I}$, which we do not write explicitly, since the exact ordering will not have any influence on the computations below, where we are going to show that the above choices satisfy Condition \ref{hyp:solving_2}.

    \begin{enumerate}
        \item The set of matrices $\{B^k\}_{\substack{k=1}}^{p}$ are such that
\begin{equation}
B^k f_I = A^{(k,I)} f_I \overset{\eqref{eq:ext_def2}}{=} e_k
\qquad
\forall k \in \{1, \dots, p\}
\end{equation}
and, if $k\in\{p+1,\dots,n\}$, then $B^k=A^J$ for some $J\not\supset I$, hence
\begin{equation}
B^kf_I= 0 \qquad
\forall k \in \{p+1, \dots, n\}.
\end{equation}
Hence the matrices $\{B^k\}_{k=1}^{n}$ satisfy point \ref{pnt:hyp_suff_1} of Condition \ref{hyp:solving_2} with
\begin{equation}
    (B^1 f, \dots B^p f, B^{p+1}f \dots, B^n f)=(e_{1},\dots,e_p,0,\dots,0).
\end{equation}

\item We now prove that the conditions of point \ref{pnt:hyp_suff_2} are satisfied.

\begin{enumerate}
\item 
Since
\begin{equation}
B^k g_k = A^{(k,I)}f_{(k,p+1,p+2,\dots,d-1)} \overset{\eqref{eq:ext_def_A_J}}{=} (-1)^{d-p} e_d = e
\qquad
\forall k \in \{1,\dots,p\},
\end{equation}
\eqref{eq:hyp_suff_2a} is satisfied.

\item
Since $g_k=0$ for $k > p$, we only have to check the validity of \eqref{eq:hyp_suff_2b} for $k \leq p$. In this case $g_k=f_{(k,p+1,p+2,\dots,d-1)}$, hence
\begin{equation}
\begin{split}
    B^\ell g_k = A^J f_{(k,p+1,p+2,\dots,d-1)} \notin \spn(e_1,\dots,e_p)
    & \overset{\eqref{eq:ext_def_A_J}}{\iff}
    J \supset (k,p+1,p+2,\dots,d-1) \,\text{and}
    \\
    & \qquad \quad J \setminus (k,p+1,p+2,\dots,d-1) \not \subseteq \{1,\dots,p\}
    \\
    & \iff
    J = (k,p+1,p+2,\dots,d-1,d)
    \\
    & \overset{\eqref{eq:ext_choices}}{\iff}
    \ell=k.
\end{split}
\end{equation}
Thus \eqref{eq:hyp_suff_2b} holds true.

\item
It remains to show \eqref{eq:hyp_suff_2c}. In order to do so, we first notice that, if $k \geq p+1$, then $g_k=0$ and $B^kf=0$, thus
\begin{equation}
B^\l g_k \cdot B^h f
=
0
=
- B^\l g_h \cdot B^k f
\qquad
\forall \ell \in \{1,\dots,n\}.
\end{equation}
Since the same holds for $h\geq p+1$, it remains to deal with the case $1 \leq k \leq h \leq p$.

Let us assume $B^\ell=A^J$ for some $(d-p+1)$-multi-index $J$. Then 
\begin{equation}
    B^\ell g_k \cdot B^h f
    =
    A^J f_{(k,p+1,\dots,d-1)} \cdot e_h
    \overset{\eqref{eq:ext_def_A_J}}{=}
    \begin{cases}
        0                       & \text{if } J \neq (k,h,p+1,\dots,d-1)
        \\
        -1                      & \text{if } J= (k,h,p+1,\dots,d-1)
    \end{cases} 
\end{equation}
while on the other hand we have
\begin{equation}
    B^\ell g_h \cdot B^k f
    =
    A^J f_{(h,p+1,\dots,d-1)} \cdot e_k
    \overset{\eqref{eq:ext_def_A_J}}{=}
    \begin{cases}
        0                       & \text{if } J \neq (k,h,p+1,\dots,d-1)
        \\
        1                      & \text{if } J= (k,h,p+1,\dots,d-1),
    \end{cases}
\end{equation}
which completes the proof of \eqref{eq:hyp_suff_2c}.\qedhere

\end{enumerate}
    \end{enumerate}
\end{proof}

We conclude by stating the following conjecture on the sharp value of $k$ for which the exterior derivative is $k$-balanceable.

\begin{conjecture}\label{conj:ext_derivative}
The exterior derivative acting on $(d-q)$-forms is $k$-balanceable if and only if $k \in \{q,\dots,d-1\}$.
\end{conjecture}

\begin{remark}
    Since $\Ac$ is represented by matrices of rank $d-q+1$, part \ref{item:main_sharpness} of Theorem \ref{thm:main_scalar} implies that it is not $k$-balanceable for $k< d+1-(d-q+1)=q$.
\end{remark}

\subsection{Comparing scalar and vector case}\label{sec:scalar_comparing_vector}
It is worth spending a few words on the relation between Theorem \ref{thm:main_scalar} and Theorem \ref{thm:main_vector_valued}. We start with the following remarks.
\begin{itemize}
    \item Theorem \ref{thm:main_scalar} is sharp in the sense that the condition $\rank A \geq 2$ is necessary and sufficient to positively answer Question \ref{prb:alberti} for scalar operators; moreover Theorem \ref{thm:main_scalar} sharply characterizes ``how singular" the measure $\sigma$ can be in order to have $\Ac(f +\sigma)=0$, depending on $\rank A$.
    \item On the other hand, Condition \ref{hyp:solving_2} is sufficient to positively answer Question \ref{prb:alberti} and the measure $\sigma$ built in the proof of Theorem \ref{thm:main_vector_valued} is $(d-1)$-rectifiable.
\end{itemize}

As mentioned in Section \ref{sec:intro}, Condition \ref{hyp:solving_2} is satisfied by any balanceable scalar-valued operator:
\begin{proposition}\label{prop:scalar_condition_hyp}
    Let $\Ac=(A) \in \op{d,m,1}$ be a scalar operator. Then $\Ac$ satisfies Condition \ref{hyp:solving_2} if and only if $\rank A=0$ or $\rank A \geq 2$.
\end{proposition}

\begin{proof}
        If $\rank A=0$, then $\Ac$ is the null-operator, which clearly satisfies Condition \ref{hyp:solving_2} for every $f \in \R^m$ by choosing $p=0$ and $g_1=0$.
        Let us assume $\rank A \geq 2$, let us fix any basis $\mathcal{F}=\{f_1,\dots,f_m\}$ of $\R^m$ and $f \in \mathcal{F}$. If $A f=0$, then the properties defined in Condition \ref{hyp:solving_2} are trivially satisfied by choosing $p=0$ and $g_1 =0 \in \R^m$. Thus we can assume that $A f \neq 0$; then, for some $\lambda \in \R\setminus\{0\}$,
        \begin{equation}
            \left( {\lambda} A \right) f\eqqcolon e_1 \in \Sph^{d-1}.
        \end{equation}
        Moreover, since $\rank A \geq 2$, there exists $g_1 \in \R^m$ such that
    \begin{equation}
        e_1\perp\left( {\lambda} A \right) g_1 \eqqcolon e \in \Sph^{d-1}.
    \end{equation}
    Setting $B \coloneqq {\lambda} A$, these are precisely part \ref{pnt:hyp_suff_1} and \eqref{eq:hyp_suff_2a} of Condition \ref{hyp:solving_2}. Since $n=1$, \eqref{eq:hyp_suff_2b} is clearly void, while \eqref{eq:hyp_suff_2c} follows from $e_1 \cdot e=0$.

    Conversely, let us assume that $\Ac=(A)$ satisfies Condition \ref{hyp:solving_2} for some basis $\Fc$ of $\R^m$. If $Af=0$ for every $f \in \Fc$, then $A=0$; otherwise let us fix $f \in \Fc$ such that $Af \neq 0$. In this last case, choosing $p=1$, there exists $B=\lambda A$ which satisfies the properties of Condition \ref{hyp:solving_2} for $f$. In particular, $Bf \eqqcolon e_1 \in \Sph^{d-1}$ and there exists $g_1 \in \R^m$ such that $B g_1\eqqcolon e \in \Sph^{d-1} \cap (Bf)^\perp$ by \eqref{eq:hyp_suff_2a} and  \eqref{eq:hyp_suff_2c}. This implies $\rank A \geq 2$.
\end{proof}



\begin{remark}\label{rem:comparing_scalar_vector}
    If $\mathcal{A}=(A)$ with $\rank A \geq 2$, then $\Ac$ is $(d-1)$-balanceable by Theorem \ref{thm:main_scalar}. The same conclusion can be inferred by Theorem \ref{thm:main_vector_valued}, thanks to Proposition \ref{prop:scalar_condition_hyp}.

    It is worth noting that the $(d-1)$-rectifiable measures $\sigma$ and $\sigma'$, produced respectively by the proof of Theorem \ref{thm:main_scalar} and the proof of Theorem \ref{thm:main_vector_valued}, are different in general. Indeed, if $f=v \ind{B}$ for some $v \in \mathbb{R}^m$ and a ball $B$, then $\sigma$ may charge the interior of $B$ (depending on the coverings chosen in Proposition \ref{prop:reduction_cubes_basis}), while $\sigma'$ is supported on $\partial B$.
\end{remark}

\subsection{Links with the wave cone}\label{subsec:wavecone}
We recall the definition of the \emph{wave cone} (see \cite{dph_rindler_afree_original}) of the operator $\Ac \in \op{d,m,n}$:
\begin{equation}
    \Lambda_{\Ac}
    \coloneqq
    \bigcup_{e \in \mathbb{S}^{d-1}} \ker \mathbb{A}[e],
\end{equation}
where $\mathbb{A}$ is the Fourier symbol of the operator $\Ac$; namely, if we write the operator in the form $\mathcal{A}f = \sum_{i=1}^d M^i \partial_i f$, where $M^i \in \mathbb{R}^{n \times m}$, then the Fourier symbol $\mathbb{A}$ is the linear function
\begin{equation}\label{eq:Fourier_symbol}
   \mathbb{A} \colon \xi \in \mathbb{R}^d
   \longmapsto
   \sum_{i=1}^d \xi_i M^i \in \mathbb{R}^{n \times m}.
\end{equation}
With the notations of this paper, if $\Ac=(A^1,\dots,A^n)$, then it is easily seen that $A_{ij}^k = M_{kj}^i$. Thus, for $f \in \R^m$, it holds the equivalence
\begin{equation}\label{eq:wave_cone_caratt}
    f \in \Lambda_{\Ac} \iff
    \exists e \in \mathbb{S}^{d-1} \colon e \perp A^k f
    \quad
    \forall k \in \{1,\dots,n\}.
\end{equation}
This characterization in particular shows that the wave cone is invariant for equivalent operators, namely $\Ac \cong \mathcal{B} \implies \Lambda_{\Ac} = \Lambda_{\mathcal{B}}$.
\begin{remark}\label{prop:links_wave_cone}
    Given a first-order operator $\Ac$, assume that $f \in \R^m$ is such that there exist $\Bc \cong \Ac$, $g_1,\dots,g_n$ and $e$ satisfying the properties of Condition \ref{hyp:solving_2}. Then $f \in \Lambda_{\Ac}$.
    In particular, if $\Ac$ satisfies Condition \ref{hyp:solving_2}, then $\spn (\Lambda_{\Ac})= \R^m$.

    The converse is false: in general, even though $\spn (\Lambda_{\Ac})= \R^m$, the operator $\Ac$ may be not be $k$-balanceable for any $k \in \{1,\dots,d-1\}$. As an example, consider any scalar operator $\Ac=(A)$ with $\rank A=1$.
\end{remark}


\subsection{Operators that do not satisfy Condition \ref{hyp:solving_2} and the divergence on symmetric matrix-valued vector fields}\label{sec:negative_examples}
Some examples of operators that do not satisfy Condition \ref{hyp:solving_2} may be produced in a rather straight-forward way:
\begin{itemize}
    \item By Subsection \ref{subsec:wavecone}, any operator $\Ac$ such that $\spn\Lambda_\Ac\neq\R^m$ does not satisfy Condition \ref{hyp:solving_2}. 
    This includes elliptic operators (that is those for which $\Lambda_\Ac=\{0\}$), like $\Ac=\nabla\in\op{d,m,d\cdot m}$ and $\Ac=(\dive,\curl)\in\op{d,d,1+\frac{d(d-1)}{2}}$.
    \item By Theorem \ref{thm:main_scalar}, any scalar operator $\Ac=(A)$ with $\rank A=1$ is not $(d-1)$-balanceable, thus it cannot satisfy Condition \ref{hyp:solving_2}.
    \item By the previous point, any vector-valued operator $\Ac=(A^1,\dots,A^n)$ so that there exists $B\in\spn\{A^1,\dots,A^n\}$ with $\rank B=1$ cannot satisfy Condition \ref{hyp:solving_2}.
    \item Analogously, if $\Ac=(A^1,\dots,A^n)$ and $\spn\{A^1,\dots,A^n\}$ contains a lower-dimensional elliptic operator, for example in the case of a two-dimensional $(\dive,\curl)$
    \begin{equation}
        A^1=\begin{pmatrix}
            1&0&0\\0&1&0\\0&0&0
        \end{pmatrix},\qquad
        A^2=\begin{pmatrix}
            0&1&0\\-1&0&0\\0&0&0
        \end{pmatrix},
    \end{equation}
    then $\Ac$ is not $k$-balanceable, thus it cannot satisfy Condition \ref{hyp:solving_2}. 
\end{itemize}

Another case of interest for Question \ref{prb:alberti} is given by the divergence acting on vector fields valued in $\mathbb{R}_{\mathrm{sym}}^{d \times d}$, the space of symmetrix ($d\times d$)-matrices. More precisely this operator is $\Ac \colon C^{\infty}(\R^d,\mathbb{R}_{\mathrm{sym}}^{d \times d}) \to C^{\infty}(\R^d,\R^d)$ acting as
\begin{equation}
    \dive
    \begin{pmatrix}
        f_{11}  & f_{12} & \cdots & f_{1d}
        \\
        f_{12}  & f_{22} & \cdots & f_{2d}
        \\
        \vdots  & \vdots & \ddots & \vdots
        \\
        f_{1d}  & f_{2d} & \cdots & f_{dd}
    \end{pmatrix}
    =
    \begin{pmatrix}
        \dive(f_{11}, f_{12}, \cdots, f_{1d})
        \\
        \dive(f_{12}, f_{22}, \cdots, f_{2d})
        \\
        \vdots
        \\
        \dive(f_{1d}, f_{2d}, \cdots, f_{dd})
    \end{pmatrix}.
\end{equation}
According to Question \ref{prb:alberti}, one may ask whether the operator is $k$-balanceable for some $k$, namely if to any map in $L^1(\R^d,\mathbb{R}_{\mathrm{sym}}^{d \times d})$ it is possible to add a $k$-rectifiable measure with values in $\mathbb{R}_{\mathrm{sym}}^{d \times d}$ in order to make the sum divergence-free.

Notice that, even if $\mathbb{R}_{\mathrm{sym}}^{d \times d}$ is a subspace of $\mathbb{R}^{d \times d}$ and the divergence on $\mathbb{R}^{d \times d}$ is $k$-balanceable for every $k \in \{1,\dots,d-1\}$ by application of Theorem \ref{thm:main_scalar}, this does not imply that so is $\Ac$, since the $k$-rectifiable measure produced to balance a symmetric vector field is not symmetric in general.

Unfortunately, Theorem \ref{thm:main_vector_valued} does not give an answer to the above question, since we may prove that $\Ac$ does not satisfy Condition \ref{hyp:solving_2} for $d=2$.
Although (to the best of our knowledge) the question remains open, we think that the following result is of independent interest, as an example of a non-trivial operator (as the ones presented earlier in this section may be) that does not satisfy Condition \ref{hyp:solving_2}.
\begin{proposition}\label{prop:sym_dive}
        The divergence operator $\Ac \colon C^{\infty}(\R^2,\mathbb{R}_{\mathrm{sym}}^{2 \times 2}) \to C^{\infty}(\R^2,\R^2)$ does not satisfy Condition \ref{hyp:solving_2}.
\end{proposition}
\begin{proof}
        We observe that $\dim \mathbb{R}_{\mathrm{sym}}^{2 \times 2}= 3$, thus $\Ac \in \op{2,3,2}$, namely $m\coloneqq 3$ and $n=2$.  We order the components of the symmetric matrix by rows and we call $\{e_1,e_2\}$ the standard orthonormal basis of $\R^2$ and $\{f_1,f_2,f_3\}$ the standard orthonormal basis of $\R^3$.
    It follows that $\Ac=(A^1,A^2)$,  where $A^k \in \mathbb{R}^{2 \times 3}$ are the following matrices:
    \begin{equation}
        A^1 = \begin{pmatrix}
            1           &       0   &       0
            \\
            0           &       1   &       0
        \end{pmatrix},
        \qquad
        A^2 = \begin{pmatrix}
            0           &       1   &       0
            \\
            0           &       0   &       1
        \end{pmatrix}.
    \end{equation}
    Let us assume, by contradiction, that there exists a basis $\mathcal{F}$ of $\R^3$ for which Condition \ref{hyp:solving_2} is satisfied and let us fix $f \in \mathcal{F}$ with non-trivial component in the direction $f_2$, namely $f=\alpha_1 f_1 + \alpha_2 f_2 + \alpha_3 f_3$ with $\alpha_2 \neq 0$. Let us fix $\mathcal{B}=(B^1,B^2) \cong \mathcal{A}$, $e$ and $g_1$ satisfying the properties of Condition \ref{hyp:solving_2}.
    
    We first remark that \eqref{eq:hyp_suff_2a}, \eqref{eq:e_perp_Bf} and $d=2$ imply
    \begin{equation}
        \dim \Big(\spn(B^1 f, B^2 f)\Big)=\dim \Big( \spn(A^1 f,A^2f) \Big)\leq 1,
    \end{equation}
    while $A^1 f= \alpha_1 e_1 + \alpha_2 e_2$ gives $\dim \big(\spn(A^1 f, A^2 f)\big) = 1$. By simple computations this implies
    \begin{equation}
        f \in \spn(f_1+\alpha f_2 + \alpha^2 f_3)
        \qquad
        \text{for some $\alpha \in \R \setminus\{0\}$.}
    \end{equation}
    Without loss of generality, we can assume $f=f_1+\alpha f_2 + \alpha^2 f_3$ and
    \begin{equation}
        B^1=\lambda_1 A^1 + \mu_1 A^2,
        \qquad
        B^2=\lambda_2 A^1 + \mu_2 A^2.   
    \end{equation}
    Since $B^1 f \in \spn(A^1 f)= \spn(e_1 + \alpha e_2)$ and $|B^1 f|=1$, we have
    \begin{equation}\label{eq:sym_dive_contr}
    \begin{split}
        (1+\alpha^2)^{-\frac{1}{2}}(e_1+\alpha e_2)=B^1 f
    & \iff
    (1+\alpha^2)^{-\frac{1}{2}}(e_1+\alpha e_2) = (\lambda_1 A^1 + \mu_1 A^2)(f_1+\alpha f_2 + \alpha^2 f_3)
   \\
   & \iff
   (1+\alpha^2)^{-\frac{1}{2}} = \lambda_1 + \mu_1 \alpha.
    \end{split}
    \end{equation}
    On the other hand, since $B^2 f=0$, it holds
    \begin{equation}
    \begin{split}
        0 = B^2 f
        & \iff
        0 = (\lambda_2 A^1 + \mu_2 A^2)(f_1 + \alpha f_2 + \alpha^2 f_3)
        =
        (\lambda_2 + \mu_2 \alpha)e_1 + \alpha(\lambda_2 + \mu_2 \alpha)e_2
        \\
        & \iff
        \lambda_2 =- \mu_2 \alpha.
    \end{split}
    \end{equation}
    Without loss of generality we can choose $\mu_2=1$, thus obtaining $B^2=-\alpha A^1 + A^2$.
    \eqref{eq:hyp_suff_2b} implies $B^2 g_1 \perp B^1 g_1$ and \eqref{eq:cons_1_hyp_solving_2} gives $B^2 g_1 \perp B^1 f$; since $\dim \big( \spn (B^1 g_1, B^1 f)\big)=2$, it must hold $B^2 g_1=0$. Hence, if $g_1=c_1f_1+c_2f_2+c_3f_3$, we have
    \begin{equation}
        0 =(-\alpha A^1+A^2)(c_1f_1+c_2f_2+c_3f_3)
        =
        -\alpha c_1 e_1 - \alpha c_2 e_2 + c_2 e_1 + c_3 e_2
        \iff
        (c_1,c_2,c_3)=(c,\alpha c, \alpha^2 c).
    \end{equation}
    On the other hand it holds
    \begin{equation}
        B^1 g_1
        =
        c(\lambda_1 A^1+ \mu_1 A^2)(f_1 + \alpha f_2 +\alpha^2 f_3)
        =
        c(\lambda_1 + \alpha \mu_1)e_1 + c \alpha(\lambda_1 + \alpha \mu_1)e_2.
    \end{equation}
     Since Condition \ref{hyp:solving_2} implies $B^1 g_1 \perp B^1 f= e_1 + \alpha e_2$, the above equation yields
     \begin{equation}
          (\lambda_1 + \alpha \mu_1) + \alpha^2 (\lambda_1 + \alpha \mu_1)=0,
     \end{equation}
     which in turn implies $(\lambda_1 + \alpha \mu_1)=0$, since $\alpha \neq 0$, but this contradicts \eqref{eq:sym_dive_contr}.
\end{proof}

\subsection{Another operator that satisfies Condition \ref{hyp:solving_2}}\label{sec:positive_example}
We conclude this collection of remarks by providing an example of vector-valued operator which satisfies Condition \ref{hyp:solving_2}; we notice that the matrices defining $\Ac$ have no evident antisymmetry.
\begin{example}\label{ex:vector_valued_condition}
    Let us consider the operator $\Ac=(A_1,A_2) \in \op{3,5,2}$, where
\begin{equation}
A^1
=
\begin{pmatrix}
1	& 	0	&	0	&	0	&	0\\
0	&	0	& 	0	&	0	&	1\\
0   &   0   &   1   &   0   &   0
\end{pmatrix}
,
\qquad
A^2
=
\begin{pmatrix}
0	& 	0	&	0	&	1	&	0\\
1	&	0	& 	0	&	0	&	0\\
0   &   1   &   0   &   0   &   0
\end{pmatrix}
\end{equation}
and let $e_1,e_2,e_3$ and $f_1,\dots,f_5$ be the vectors of the standard orthonormal bases respectively of $\mathbb{R}^3$ and $\mathbb{R}^5$.
We claim that $\Ac$ satisfies Condition \ref{hyp:solving_2} with the basis $\mathcal{F}\coloneqq \{f_1,f_5,f_2,f_3,f_4,f_5\}$.

\begin{itemize}
    \item It easy to see that $f_2,\dots,f_5$ satisfy properties of Condition \ref{hyp:solving_2}; for instance, for $f=f_2$, one can choose $B^1=A^2$, $B^2= A^1$, $g_1=f_4$, $g_2 = 0$ and $e=e_1$. The others are similar.
    \item For $f=f_1$, then the choices $B^1 = A^1$, $B^2 = A^2$, $g_1=f_3$, $g_2=f_2$ and $e=e_3$ work.
\end{itemize}
For the sake of completeness, we observe that the wave cone of $\mathcal{A}$ is $\mathbb{R}^5$ (as one can see by \eqref{eq:wave_cone_caratt} since $d>n$).
\end{example}

\section{A Lusin-type property}\label{sec:lusin}

In this section we prove Theorem \ref{thm:Lusin_main}.
The proof is based on Lemma \ref{lmm:lusin} below. With this result at hand, the proof of \citep[Theorem 1]{alberti1991lusin} can be repeated without any modification, except for replacing $Du$ with $h$ such that $\Ac h=0$.

\begin{lemma}\label{lmm:lusin}
	Let $\Omega \subset \R^d$ be an open set with finite measure and let $\Ac \in \op{d,m,n}$ be a first order operator satisfying Condition \ref{hyp:solving_2}. For every continuous and bounded function $f:\Omega\to\R^m$ and every $\varepsilon,\eta>0$, there exist a compact set $K$ and $h\in C^\infty_c(\Omega;\R^m)$ such that	\noeqref{eq:lusin_lemma_a,eq:lusin_lemma_b,eq:lusin_lemma_c,eq:lusin_lemma_d}
	\begin{subequations}\label{eq:lusin_lemma}
		\begin{gather}
			\Lc^d(\Omega \setminus K) < \varepsilon \Lc^d(\Omega);\label{eq:lusin_lemma_a}
			\\[2pt]
			|f-h|< \eta \qquad \text{on } K;\label{eq:lusin_lemma_b}
			\\[2pt]
			\Ac h=0\qquad\text{in }\Omega;\label{eq:lusin_lemma_c}
			\\[2pt]
			\norm{h}_{L^p(\Omega)} \leq C_1 \varepsilon^{\frac{1}{p}-1} \norm{f}_{L^p(\Omega)}\label{eq:lusin_lemma_d} \qquad
   \forall p \in [1,+\infty],
		\end{gather}
  where the constant $C_1$ depends only on $\Ac$.
	\end{subequations}
\end{lemma}

\begin{proof}
	Let $\mathcal{F}\coloneqq\{f_1,\dots,f_m\}$, $B^1,\dots,B^n$, $e_1,\dots,e_p$ and $e\coloneqq e_{p+1}$ be as in Condition \ref{hyp:solving_2}.
	Up to changing coordinates, we may assume that $e_1,\dots,e_{p+1}$ are the first $p+1$ vectors of the standard orthonormal basis of $\R^d$ and that $|f_i|=1$ for all $f_i\in\mathcal{F}$.
    
	\begin{description}
	    \item[$\bullet$ Step 1] We first show that, given any $v\in\mathcal{F}$ and $\beta>0$, we may find $u\in C^\infty_c(B_1,\R^m)$ such that
	\begin{gather}
		u\equiv v\mbox{ in }B_{1-\beta},\label{eq:u_coincides_v}\\
		\Ac u=0,\label{eq:h_is_afree}\\
		||u||_{L^p(B_1)}\le C_1 \beta^{1/p-1}
  \qquad
  \forall p \in [1,+\infty],\label{eq:stima_su_h}
	\end{gather}
	where $C_1$ is a constant depending only on $\Ac$.
	
	We proceed as follows. By Proposition \ref{prop:hyp_sufficient_solving}, there exists $g:\de B_1\to\R^m$ such that
	\begin{equation}\label{eq:solve_for_v}
		\Ac(v\ind{B_1}+g\Hc{d-1}\rest\de B_1)=0.
	\end{equation}
	We define
	\begin{equation}
		u(x) \coloneqq \phi(|x|)v -|x|\phi'(|x|) g\left( \frac{x}{|x|} \right),
	\end{equation}
	where $\phi\in C^\infty_c([0,1))$ is such that $\phi\equiv1$ on $[0,1-\beta]$ and $||\phi'||_{L^\infty}\le \frac{2}{\beta}$.
    Notice that $u \equiv v$ on $B_{1-\beta}$ and $u\in C_c^\infty(B_1,\R^m)$. In order to prove \eqref{eq:h_is_afree}, by Lemma \ref{lmm:Amu_dive} we need to show that $\dive (B^k u)=0$ for all $k \in \{1,\dots,n\}$. 
    We first consider any test function $\chi\in C^\infty_c(\R^d\setminus\{0\})$ and we compute
	\begin{equation}\label{eq:dive_Bg}
		\begin{split}
			\iprod{\dive \left( |x|B^k g\left( \frac{x}{|x|} \right) \right)}{\chi}
			&=
			-\iprod{|x|B^k g\left( \frac{x}{|x|} \right)}{\nabla \chi}
			\\
			&=
			- \int_0^{+\infty} t\int_{\de B_t} \nabla \chi(x) \cdot B^k g\left( \frac{x}{t} \right) \dif \Hc{d-1}(x) \dif t
			\\
			&=
			- \int_0^{+\infty} t^{d} \int_{\de B_1} \nabla \chi(tx) \cdot B^k g\left( x\right) \dif \Hc{d-1}(x) \dif t
			\\
			&=
			\int_0^{+\infty} t^{d-1} \iprod{\dive \big( B^kg \Hc{d-1} \rest \de B_1 \big)}{\chi(t\, \cdot)} \dif t
			\\
			\overset{\eqref{eq:solve_for_v}}&{=}
			\int_0^{+\infty} t^{d-1} \int_{\de B_1} \big( B^k v \cdot \nu_{B_1}(x)\big) \chi(t x) \dif \Hc{d-1}(x)\dif t
			\\
			&=
			\int_0^{+\infty} \int_{\de B_t} \big( B^k v \cdot \nu_{B_t}(x)\big) \chi(x) \dif \Hc{d-1}(x)\dif t
			\\
			&=
			\iprod{B^k v \cdot \frac{x}{|x|}}{\chi}.
		\end{split}
	\end{equation}
	We therefore have
 	\begin{align}
		\dive(B^k u(x))
  =
  \phi'(|x|)\frac{x}{|x|}\cdot B^kv   
  -|x|\phi''(|x|)\frac{x}{|x|}\cdot B^kg\left(\frac{x}{|x|}\right) - \phi'(|x|)B^kv\cdot\frac{x}{|x|}
 =
  0,
	\end{align}
	where in the second equality the fact that $B^k g \perp \nu_{B_1}$ (by \eqref{eq:Ag_perp_face}) was used.
	Thus $\Ac u=0$ by Lemma \ref{lmm:Amu_dive} and \eqref{eq:h_is_afree} is proven.
	In order to prove \eqref{eq:stima_su_h}, we use the fact that $\norm{g}_\infty\le C$ for some $C$ depending only on $\Ac$, hence on $B_1\setminus B_{1-\beta}$ it holds
	\begin{align*}
		|u(x)|\le C'(1+|\phi'(x)|)\le \frac{C''}{\beta}
	\end{align*}
	for some $C',C''$ large enough depending only on $\Ac$. In particular, 
	\begin{equation*}
		\norm{u}_{L^p(B_1)}\le C_1\beta^{1/p-1}
  \qquad
  \forall p \in [1,+\infty],
	\end{equation*}
    where $C_1$ depends only on $\Ac$.

	\item[$\bullet$ Step 2] We now proceed with the proof of the result.	
	Arguing component-wise, we may assume without loss of generality that there exists $v\in\mathcal{F}$ such that $f(x)\in\spn v$ for every $x$.
    Let us fix a compact set $K' \subset \Omega$ such that $\Lc^d(\Omega \setminus K')< \frac{\varepsilon}{2}\Lc^d(\Omega)$.
    By uniform continuity of $f$ on compact subsets of $\Omega$ there exists $\delta>0$ such that, if $x \in K'$, then $B_{2\delta}(x) \subset \Omega$ and
	\begin{equation}\label{eq:uniform_continuity_cubes}
		|f(y)-f(z)|< \eta \qquad
		\forall y,z \in B_{\delta}(x).
	\end{equation}
    By Vitali's covering theorem (\citep[Theorem 2.8.17]{federer2014geometric}), there exists a countable collection of mutually disjoint open balls $\Fc\coloneqq\{B_i\coloneqq B_{r_i}(x_i)\}_{i\in\N}$ with $r_i\in(0,\delta]$ and $x_i\in K'$ for every $i$ such that
    \begin{equation}
        \Lc^d\left(K'\setminus\bigcup_{B_i\in\Fc}B_i\right)=0.
    \end{equation}
    Given $B_i=B_{r_i}(x_i)\in\Fc$, we let $\hat{B}_i\coloneqq \overline{B_{(1-\beta)r_i}(x_i)}$, where $\beta\coloneqq\frac{\varepsilon}{2 d} \in (0,1)$, and we define
    \begin{equation}
        K\coloneqq \bigcup_{i\in\N}\hat{B}_i.
    \end{equation}
    Notice that, by definition of $\beta$, it holds
 \begin{equation}
     \Lc^d \left( B_1 \setminus \overline{B_{1-\beta}} \right)
     \leq
     d\beta\Lc^d(B_1)
     =
     \frac{\varepsilon}{2} \Lc^d \left( B_1 \right),
 \end{equation}
 therefore
\begin{equation}\label{eq:occupo_tanto_spazio}
\begin{split}
    \Lc^d \left( \Omega \setminus K\right)
 & \leq
\Lc^d(\Omega \setminus K') +
     \Lc^d \left( K' \setminus K\right)
     \\
     & \leq
     \frac{\varepsilon}{2}\Lc^d(\Omega )+
      \Lc^d \left( \bigcup_{i\in\N} \Big( B_i \setminus \hat{B}_i\Big) \right)
      \\
      & \leq
      \frac{\varepsilon}{2}\Lc^d(\Omega )+
      \frac{\varepsilon}{2} \Lc^d \left( \bigcup_{i\in\N}  B_i  \right)
      \\
      & \leq
      \varepsilon \Lc^d(\Omega).
\end{split}
 \end{equation}
Next, for all $B_i\in\Fc$, define
	\begin{equation}
		a_i \coloneqq \fint_{B_i} f \dif\Lc^d;
	\end{equation}
	by \eqref{eq:uniform_continuity_cubes}, for every $B_i\in\Fc$ and every $y\in B_i$ it holds
	\begin{equation}\label{eq:unif_contin_a_pezzi}
		|a_i - f(y)|< \eta.
	\end{equation}
 Let now
	\begin{equation}\label{eq:def_h_lusin}
		h(x)\coloneqq \sum_{i\in\N} (a_i \cdot v)\, u\left(\frac{x-x_i}{r_i}\right)\ind{B_i}(x),
	\end{equation}
	where $u$ was defined in Step 1. Then \eqref{eq:lusin_lemma_a} and \eqref{eq:lusin_lemma_b} follow from \eqref{eq:occupo_tanto_spazio} and \eqref{eq:unif_contin_a_pezzi}, respectively.
 Moreover, since $\supp u\subset B_1$, by linearity it holds
	\begin{equation*}
		\Ac h = 0.
	\end{equation*}
 Lastly, \eqref{eq:lusin_lemma_d} follows immediately from \eqref{eq:def_h_lusin} and \eqref{eq:stima_su_h} in the case $p=+\infty$ . For $p \in [1,+\infty)$, we have
 \begin{equation}
 \begin{split}
     \norm{h}_{L^p(\Omega)}
     \overset{\eqref{eq:def_h_lusin}}&{\leq}
     \sum_{i\in\N} |a_i| \norm{u }_{L^p}\big(\Lc^d(B_i)\big)^{\frac{1}{p}}
     \\
     & \leq
     \norm{u}_{L^p}\sum_{i\in\N}\big(\Lc^d(B_i)\big)^{\frac{1}{p}-1} \int_{B_i} |f|
     \\
     & \leq
     \norm{u}_{L^p}\sum_{i\in\N} \norm{f}_{L^p(B_i)}
     \\
    & \leq
     C_1
     \varepsilon^{\frac{1}{p}-1}
     \norm{f}_{L^p(\Omega)}
     \end{split}
 \end{equation}
 where in the third inequality we applied H\"older inequality and in the last one we used $\beta=\frac{\varepsilon}{2d}$, \eqref{eq:stima_su_h} and the triangular inequality for the $L^p$ norms: $|f| \geq |f|_{| \bigcup_i B_i} = \sum_i |f|_{| B_i} $ since the balls $B_i$ are mutually disjoint. \qedhere

 \end{description}
\end{proof}

\section*{Acknowledgments}
    The authors would like to express their gratitude to Guido De Philippis for suggesting the problem and for several illuminating conversations on this topic and to Giovanni Alberti for his interest in this work.
    
    L.D.M. is supported by PRIN project 2022PJ9EFL: ``Geometric Measure Theory: Structure of Singular Measures, Regularity Theory and Applications in the Calculus of Variations", partially supported by the \textsc{STARS - StG} project ``\textsc{QuASAR} - Question About Structure And Regularity of currents" and by \textsc{INDAM-GNAMPA}.
    
	C.G. is supported by the European Research Council (\textsc{ERC}), under the European Union's Horizon 2020 research and innovation program, through the project \textsc{ERC VAREG} - Variational approach to the regularity of the free boundaries (grant agreement No. 853404) and partially supported by \textsc{INDAM-GNAMPA}.

    The authors would like to thank the anonymous reviewers for their useful comments and suggestions, which greatly strengthened the manuscript.


\bibliography{biblio}

\end{document}